\documentclass[a4paper, reqno]{amsart}

\usepackage[T1]{fontenc}
\usepackage{tikz}

\usepackage{amsmath,amsthm,amssymb,relsize}
\usepackage[nobysame,alphabetic]{amsrefs}
\usepackage[margin=31mm]{geometry}
\usepackage{mathrsfs,mathtools,makecell}

\usepackage{booktabs, multirow, adjustbox, array}

\newcolumntype{R}[2]{%
    >{\adjustbox{angle=#1,lap=\width-(#2)}\bgroup}%
    l%
    <{\egroup}%
}

\usepackage{caption}

\usepackage{enumerate}
\usepackage{tikz}
\usetikzlibrary{
  cd,
  calc,
  positioning,
  fit,
  arrows,
  decorations.pathreplacing,
  decorations.markings,
  shapes.geometric,
  backgrounds,
  bending
}
\usepackage{tikzsymbols}
\PassOptionsToPackage{dvipsnames}{xcolor}
\definecolor{darkblue}{rgb}{0,0,0.6}

\usepackage[breaklinks, pdftex, ocgcolorlinks,colorlinks=true, citecolor=darkblue, filecolor=darkblue, linkcolor=darkblue, urlcolor=black]{hyperref}
\usepackage[capitalize,noabbrev]{cleveref}

\makeatletter
\newtheorem*{rep@theorem}{\rep@title}
\newcommand{\newreptheorem}[2]{%
\newenvironment{rep#1}[1]{%
 \def\rep@title{#2 \ref{##1}}%
 \begin{rep@theorem}}%
 {\end{rep@theorem}}}
\makeatother

\newtheorem{proposition}{Proposition}[section]
\newtheorem{theorem}[proposition]{Theorem}
\newtheorem{corollary}[proposition]{Corollary}
\newtheorem{lemma}[proposition]{Lemma}
\newtheorem{scholium}[proposition]{Scholium}

\newtheorem{thmx}{Theorem}

\theoremstyle{definition}
\newtheorem{definition}[proposition]{Definition}

\theoremstyle{remark}
\newtheorem{remark}[proposition]{Remark}

\newtheorem*{remark*}{Remark}

\newreptheorem{theorem}{Theorem}
\newreptheorem{lemma}{Lemma}
\newreptheorem{proposition}{Proposition}
\newreptheorem{corollary}{Corollary}
\newreptheorem{question}{Question}
\numberwithin{equation}{section}

\newcommand{\Diff}{\mathrm{Diff}}

\newcommand{\sign}{\operatorname{sign}}

\newcommand{\N}{\mathbb{N}}

\newcommand{\R}{\mathbb{R}}
\newcommand{\Z}{\mathbb{Z}}

\newcommand{\A}{\mathcal{I}}

\newcommand{\PD}{PD}
\newcommand{\im}{\operatorname{Im}}
\newcommand{\Id}{\operatorname{Id}}

\newcommand{\ol}{\overline}
\newcommand{\wt}{\widetilde}

\newcommand{\sm}{\setminus}

\newcommand{\ks}{\operatorname{ks}}
\newcommand{\CP}{\mathbb{CP}}

\DeclareMathOperator{\Int}{Int}

\DeclareMathOperator{\BTOP}{BTOP}
\DeclareMathOperator{\BB}{B}
\DeclareMathOperator{\TOP}{TOP}
\DeclareMathOperator{\PL}{PL}
\DeclareMathOperator{\BO}{BO}
\DeclareMathOperator{\OO}{O}

\DeclareMathOperator{\pr}{pr}

\DeclareMathOperator{\coker}{coker}
\DeclareMathOperator{\colim}{colim}
\DeclareMathOperator{\std}{std}

\newcommand{\bsm}{\left(\begin{smallmatrix}}
\newcommand{\esm}{\end{smallmatrix}\right)}

\begin{document}
\title{Smoothing 3-manifolds in 5-manifolds}

\author{Michelle Daher}
\address{Mathematics Institute, University of Warwick, United Kingdom}
\email{michelle.daher@warwick.ac.uk}

\author{Mark Powell}
\address{School of Mathematics and Statistics, University of Glasgow, United Kingdom}
\email{mark.powell@glasgow.ac.uk}

\def\subjclassname{\textup{2020} Mathematics Subject Classification}
\expandafter\let\csname subjclassname@1991\endcsname=\subjclassname
\subjclass{
57K10, 
57N35, 
57N70. 
}
\keywords{Smoothing submanifolds, concordance of surfaces. 
}

\begin{abstract}
We show that every locally flat topological embedding of a 3-manifold in a smooth 5-manifold is homotopic, by a small homotopy, to a smooth embedding. We deduce that topologically locally flat concordance implies smooth concordance for smooth surfaces in smooth 4-manifolds.
\end{abstract}

\maketitle

\section{Introduction}

Let $Y^3=Y_1\sqcup \cdots \sqcup Y_m$ be a compact 3-manifold with connected components $Y_i$, and let $N^5$ be a compact, connected, smooth 5-manifold. Note that $Y$ and $N$ are possibly nonorientable and can have nonempty boundary. Since $Y$ is 3-dimensional it admits a unique smooth structure up to isotopy~\cite{Moise52},~\cite{Munkres-smoothing}*{Theorem~6.3},~\cite{WhdJ1961a}*{Corollary 1.18}.

\begin{thmx}\label{thm:main-intro}
Let $f\colon Y\to N$ be a locally flat proper topological embedding that is smooth near~$\partial Y$. Then $f$ is homotopic rel.\ boundary, via an arbitrarily small homotopy, to a smooth embedding.
\end{thmx}

Here \emph{proper} means that $f^{-1}(\partial N) = \partial Y$.
It is not possible in general to isotope $f$ to a smooth embedding, so the homotopy in the theorem is necessary. For instance, Lashof~\cite{Lashof} constructed a locally flat knot $L \cong S^3 \subseteq S^5$ that is not isotopic, in fact not even concordant, to any smooth knot.
We will make crucial use of Lashof's knot in our proof of \cref{thm:main-intro}.

In the rest of the introduction, we explain an application to concordance of surfaces, then we compare with the situation for codimension two embeddings in other dimensions and categories, before finally outlining our proof of \cref{thm:main-intro}.

\subsection{Topological concordance implies smooth concordance for surfaces in 4-manifolds}
Let $\Sigma$ be a closed, smooth surface, possibly disconnected, and possibly nonorientable.  We consider a smooth, closed, connected 4-manifold $X$, again possibly nonorientable, and two smooth submanifolds~$\Sigma_0$ and~$\Sigma_1$ in $X$ with $\Sigma_0 \cong \Sigma \cong \Sigma_1$.

\begin{definition}
We say that $\Sigma_0$ and $\Sigma_1$ are \emph{topologically concordant} (respectively \emph{smoothly concordant}) if there is a locally flat (respectively smooth) submanifold $C \cong \Sigma \times I$,  properly embedded in~$X \times I$, whose intersection with $X \times \{0,1\}$ is precisely $\Sigma_0 \subseteq X \times \{0\}$ and $\Sigma_1 \subseteq X \times \{1\}$. We call~$C$ a \emph{topological concordance} (respectively \emph{smooth concordance}).
\end{definition}

\begin{corollary}\label{corollary:concordance}
 Suppose that $C$ is a topological concordance between $\Sigma_0 \subseteq X\times \{0\}$ and $\Sigma_1 \subseteq X \times \{1\}$.  Then the inclusion map $C \to X \times I$ is homotopic rel.\ $\Sigma_0 \cup \Sigma_1$, via an arbitrarily small homotopy, to an embedding whose image is a smooth concordance between $\Sigma_0$ and $\Sigma_1$.
\end{corollary}

This follows immediately from ~\cref{thm:main-intro} by taking $Y = \Sigma \times I$, $N= X \times I$, and $f \colon Y \to N$ to be an embedding with $C=f(Y)$.

Special cases of \cref{corollary:concordance} were known before.
First, Kervaire \cite{Kervaire-slice-knots} proved that every 2-knot is slice. This holds in both categories, from which it follows that smooth and topological concordance coincide for 2-knots.
Sunukjian~\cite{Sunuk} proved more generally that homologous connected surfaces in a simply-connected 4-manifold $X$ are both smoothly and topologically concordant. Again, it follows immediately that smooth and topological concordance coincide.
Similarly Cha-Kim~\cite{Cha-Kim}*{Corollary~J} proved that smooth and topological concordance coincide for smoothly embedded spheres with a common smoothly embedded geometrically dual framed sphere.

Work defining surface concordance obstructions includes \cites{Stong-uniqueness, Schwartz-I, Klug-Miller, ST-FQ, AMY}. Other than in \cite{Stong-uniqueness}, the authors restricted to the smooth category. Our result implies that one automatically obtains topological concordance obstructions.

\subsection{Comparison with other dimensions}

We start with low dimensions.
In dimension 3, every locally flat embedding $Y^1 \subseteq N^3$ is isotopic to a smooth embedding. On the other hand, in dimension 4, the existence of topologically slice knots that are not smoothly slice implies the existence of a locally flat embedding $D^2 \hookrightarrow D^4$ that is not even homotopic rel.\ boundary to a smooth embedding. There are also examples of closed locally flat surfaces in closed 4-manifolds, in particular in $S^2 \times S^2$, $\CP^2$, and $S^2 \wt{\times} S^2$, that cannot be smoothed up to homotopy~\cites{Kuga, Rudolph, Luo, LW-90}.
We deal with dimension 5 in this article. The analogue of \cref{thm:main-intro} for locally flat embeddings of smooth 4-manifolds embedded in smooth 6-manifolds is open, and we intend to investigate it in future work.

Now we discuss high dimensions.
For codimension 2 proper embeddings $f \colon Y^m \to N^{m+2}$, when the dimension $m$ of $Y$ is greater or equal to 5, Schultz~\cite{SchultzSmoothableSO} proved the following cf.\ \cite{Lashof-Rothenberg}.

\begin{theorem}[Schultz]\label{thm:Schultz1}
Let $m \geq 5$ and $n>m$.  Let $N^n$ be a smooth compact $n$-manifold, and let~$Y^m$ be a compact topological manifold equipped with a smooth structure near $\partial Y$.
Let $f\colon Y\to N$ be a locally flat proper topological embedding that is smooth near~$\partial Y$.
Then there is a smooth structure on $Y$, extending the given smooth structure on $\partial Y$, such that $f$ is isotopic rel.\ boundary to a smooth embedding if and only if~$Y$ has a topological vector bundle neighbourhood.
\end{theorem}

Topological vector bundle neighbourhoods always exist for locally flat codimension 1 and 2 embeddings~\cites{Brown62,KS-normal-bundles-codim-2}, so Schultz~\cite{SchultzSmoothableSO} deduced the following result.

\begin{theorem}[Schultz]\label{thm:Schultz2}
Let $k=1$ or $2$, let $m\geq5$, and let $n=m+k$. Let $N^{n}$ be a smooth compact $n$-manifold,
and let~$Y^m$ be a compact topological manifold equipped with a smooth structure near $\partial Y$.
Let $f\colon Y\to N$ be a locally flat proper topological embedding that is smooth near~$\partial Y$.
Then there is a smooth structure on $Y$, extending the given smooth structure on $\partial Y$, such that $f$ is isotopic rel.\ boundary to a smooth embedding.
\end{theorem}

Note that in the statements of \cref{thm:Schultz1,thm:Schultz2}, $Y$ is not a priori smoothable. The existence of a topological vector bundle neighbourhood for an embedding $f \colon Y\to N$ guarantees a smooth structure on $Y\times\R^p$ for some $p\in \N\setminus\{0\}$. Hence, for $m\geq5$, the Product Structure Theorem~\cite{Kirby-Siebenmann:1977-1}*{Essay~I} implies that $Y$ is smoothable. The proofs of \cref{thm:Schultz1,thm:Schultz2} then proceed by using smoothing theory and the Concordance Implies Isotopy Theorem~\cite{Kirby-Siebenmann:1977-1}*{Essay~I} for smooth structures on $Y$.

The results from \cite{Kirby-Siebenmann:1977-1} do not apply in the same way for $Y^3 \subseteq N^5$. Our approach is rather different.  We fix a smooth structure on a tubular neighbourhood of $f(Y)$ and try to extend it to all of $N$. As we will describe in \cref{subsec:outline-of-proof}, we face obstructions along the way that will require us in general to modify the embedding by a small homotopy to obtain a smooth embedding of $Y$ in $N$.

Returning to the high dimensional case, if one first fixes a smooth structure on $Y^m$, $m \geq 5$, then the induced structure on $Y$ that emerges from Schultz's argument need not be isotopic to the fixed one.
This is a feature of the problem, not a failure of the proof. In fact, if we fix  a smooth structure on $Y$, in general in high dimensions~$f$ is not even homotopic to a smooth embedding.
Kervaire~\cite{Kervaire-higher-dim-knots}*{Appendix,~Theorem~I} (cf.~\cite{Hsiang-Levine-Szczarba} for $d=16$) showed that an exotic $n$-sphere $\Sigma^{n}$ embeds smoothly in $S^{n+2}$ if and only if $\Sigma^n$ bounds a parallelisable $(n+1)$-manifold. Thus by Kervaire-Milnor~\cite{Kervaire-Milnor:1963-1} (see e.g.~\cite{CLM}*{Theorem 12.1} for a succinct statement of the computation of $\theta_n$) an exotic sphere that does not embed exists whenever $\coker (J_n)$ contains a nontrivial element with vanishing Kervaire invariant. For example, this holds for every even $8 \leq n \leq 138$ other than $12$ and $56$~\cite{Behrens-Hill-Hopkins-Mahowald}*{Theorem~1.2}.

Before proceeding to the proof summary, we discuss other related results in the literature.
\begin{enumerate}
    \item
Kirby~\cite{Kirby-smoothing-lf-embeddings} proved that locally flat embeddings are isotopic to smooth embeddings in a range of dimensions, in particular with ambient dimension at least seven and codimension at least four.
\item
Wall~\cite{Wall-locally-flat-PL-smoothing-codim-2}*{Theorem~2} showed, in all dimensions, that every PL locally flat codimension two submanifold is isotopic to a smooth submanifold.
\item
Venema~\cite{Venema-1987} showed for every $n \geq 3$ that every topological embedding of an $(n-2)$-cell in the interior of a $PL$ $n$-manifold can be approximated by a $PL$ embedding. This does not apply to our setting of closed manifolds or proper embeddings, because it does not control the boundary of the cell.
\end{enumerate}

\subsection{Outline of the proof of \texorpdfstring{\cref{thm:main-intro}}{Theorem A}}\label{subsec:outline-of-proof}

For a submanifold $K$ of a manifold $X$ with an open tubular neighbourhood i.e.\ the image of an embedding of a normal bundle, denote the tubular neighbourhood by $\nu K$.  Write $\ol{\nu} K$ for the closure of the tubular neighbourhood of~$K$ in~$X$, which has the structure of a  disc bundle over~$K$. Given a closed subset~$C$ of~$X$,  \emph{a smooth structure on~$C$} will always mean a smooth structure on an open neighbourhood~$U$ of~$C$ in~$X$.

The proof of \cref{thm:main-intro} breaks naturally into two distinct steps, the outlines of which we shall explain next.
For a smooth structure $\sigma$ on a topological manifold~$X$, we write~$X_\sigma$ to specify that~$X$ is equipped with the smooth structure~$\sigma$. In what follows, we will write~$N_{\std}$ for~$N$ equipped with the given smooth structure.

\vspace{1em}

\noindent\textbf{Step 1:}  \textit{We show that $f\colon Y\to N$ is homotopic, by a small homotopy, to a smooth embedding $g\colon Y\to N_{\sigma}$, for some $\sigma$.}

\vspace{1em}

We write $M := f(Y)$ for the image of $f$.
The idea of the proof is to consider a standard smooth structure on $\ol{\nu} M$ and on $\partial N$, and then to try to extend this to all of $N$. We denote the exterior of~$M$ by  $W_f := N \sm \nu M$.  Smoothing theory (recapped in Section~\ref{section:smoothing-theory}) gives a Kirby-Siebenmann obstruction in $H^4(W_f,\partial W_f;\Z/2)$, that vanishes if and only if the smooth structure on $\partial W_f$ extends to all of $W_f$.  It turns out that this obstruction does not always vanish, but that by taking ambient connected sums of $M$ with copies of Lashof's nonsmoothable 3-knot $L \cong S^3 \subseteq S^5$ from~\cite{Lashof}, which we discuss in Section~\ref{section:lashofs-knot}, we can arrange that this obstruction vanishes. Whence $f$ is homotopic, via a small homotopy, to $g \colon Y \to N$ such that  $M' := g(Y)$ is smooth in some smooth structure~$\sigma$ on~$N$, that restricts to the given smooth structure on~$\partial N$.

\vspace{1em}

\noindent\textbf{Step 2:}  \textit{ We show that $g\colon Y\to N_{\sigma}$ is homotopic, via a small homotopy, to a smooth embedding $g'\colon Y\to N_{\std}$.}

\vspace{1em}

Smoothing theory implies that we can arrange for the smooth structure $\sigma$ on $N$ and the given smooth structure $\std$ to agree away from a tubular neighbourhood $\nu S$ of a surface $S \subseteq N$. By transversality we can assume that $g(Y)$ intersects $S$ in finitely many points, in a neighbourhood of which $M':= g(Y)$ is smooth in $\sigma$ but not in $\std$. This reduces the smoothing problem for $M'$ in $\std$ to finitely many local problems, which can be resolved using a proof analogous to Kervaire's proof \cite{Kervaire-slice-knots}*{Th\'{e}or\`{e}me~III.6} that every 2-knot is smoothly slice.  Kervaire's result was generalised by Sunukjian~\cite{Sunuk}, who showed that homologous connected surfaces in 1-connected 4-manifolds are smoothly concordant, and it is Sunukjian's arguments that apply in our situation.

\begin{remark}
  The changes to $f$ in Steps 1 and 2 can be characterised as topological isotopies, together with adding and removing local knots.
\end{remark}

\subsection*{Organisation}
In Section~\ref{section:smoothing-theory} we recap smoothing theory, prove lemmas on properties of the Kirby-Siebenmann invariant, and recall Lashof's nonsmoothable 3-knot. We prove Step 1 in Section~\ref{section:smoothing-complement}, and we prove Step 2 in Section~\ref{section:comparing-with-std}. Then in Section~\ref{section:Jae-choon-suggestions} we give conditions under which smoothing up to isotopy is possible.

\subsection*{Acknowledgements}

MP thanks the participants of a discussion group on surfaces in 4-manifolds in Le Croisic, June 2022, which brought this problem to his attention, and by extension thanks the organisers of this enjoyable conference.
We thank Sander Kupers for his interest and suggesting a citation, we are grateful to Jae Choon Cha for suggesting that we write Section~\ref{section:Jae-choon-suggestions}, and we thank the anonymous referee for helpful suggestions that led to improvements in the exposition and a simplification of the proof of \cref{lemma:finding:Y-s}.

MD was supported by EPSRC New Horizons grant EP/V04821X/2.
MP was partially supported by EPSRC New Investigator grant EP/T028335/2 and EPSRC New Horizons grant EP/V04821X/2.

\section{Smoothing theory}\label{section:smoothing-theory}

In this section we give a brief recap of smoothing theory, and recall the results we will need.
Smoothing theory was developed by Cairns~\cite{Cairns}, Munkres~\cites{Munkres-smoothing, Munkres-smoothing-II,Munkres-smoothing-III,Munkres-smoothing-IV,Munkres-2-sphere}, Milnor~\cites{Milnor-ICM,Milnor-microbundles},   Hirsch~\cite{Hirsch-smoothing}, Hirsch-Mazur~\cite{Hirsch-Mazur}, Lashof-Rothenberg~\cite{Lashof-Rothenberg}, and Cerf~\cites{Cerf-VI,Cerf-I,Cerf-II,Cerf-III,Cerf-IV,Cerf-V}, among others.
 Their goal, which they achieved to a large extent, was to understand which PL manifolds admit smooth structures, and if so how many.  The theory was extended around 1970 by Kirby and Siebenmann~\cite{Kirby-Siebenmann:1977-1} to allow one to start with a topological manifold, provided that one is not trying to understand smooth structures on a 4-manifold.
For the purposes of this article, since we work in dimensions four and five, the smooth and PL categories are interchangeable. Since it is more common nowadays to work in the smooth category, we shall also do so.

\subsection{Recap of smoothing theory}
Let $X$ be a topological $n$-manifold possibly with boundary, let~$C$ be a closed subset of $X$, and let $\sigma$ be a smooth structure on an open neighbourhood $U$ of $C$. Let $V \subseteq U$ be a smaller open neighbourhood of $C$.
Denote the set of isotopy classes of smooth structures on $X$ that agree with $\sigma$ near $C$ by $\mathcal{S}_{\Diff}(X, C,\sigma)$. We write $\BTOP(k)$ for the classifying space for topological $\R^k$ bundles, and $\BO(k)$ for the classifying space for rank~$k$ smooth vector bundles. Define $\BTOP := \colim_k \BTOP(k)$ and $\BO := \colim_k \BO(k)$, the corresponding stable bundle classifying spaces. Consider the following diagram, which is induced by the stable classifying maps of the tangent bundle of a neighbourhood $U$ of $C$ and the stable tangent microbundle of $X$:
\[\begin{tikzcd}
  & \TOP/\OO \ar[d] \\
  V \ar[r,"T_{V}"] \ar[d] & \BO \ar[d] \\
  X \ar[r,"\tau_X"'] \ar[ur,dashed] & \BTOP \ar[d] \\
  & \BB(\TOP/\OO).
\end{tikzcd}\]
Smoothing theory implies that for $n \geq 6$ or $(n=5 \ \mathrm{and}\  \partial X \subseteq C)$, isotopy classes of smooth structures on $X$ correspond to lifts $X\to \BO$ of the map $X \to \BTOP$, relative to the fixed lift on the smaller neighbourhood $V \subseteq U$. The vertical sequence is a  principal fibration, which implies that such a lift exists if and only if the composite $X \to \BTOP \to \BB(\TOP/\OO)$ is null-homotopic, and that homotopy classes of such lifts correspond to $[(X,V),(\TOP/\OO,*)]$ , homotopy classes of maps $X\to \TOP/\OO$ that send $V$ to the base point.

The main result of smoothing theory, applied to 5-manifolds, reads as follows \cite{Kirby-Siebenmann:1977-1}*{Theorem~IV.10.1}.

\begin{theorem}\label{thm:smoothing-thy-main-thm}
  Let $X$ be a $5$-dimensional topological manifold, let $C$ be a closed subset of $X$ with $\partial X\subseteq C$, and fix a smooth structure $\sigma$ on an open neighbourhood $U$ of~$C$.
  \begin{enumerate}[(i)]
    \item\label{item:1-smoothing-thy}   There is an obstruction $\ks(X,C):= \ks(X,C,U,\sigma) \in H^4(X, C;\Z/2)$ that vanishes if and only if $X$ admits a smooth structure extending the given smooth structure on some neighbourhood $V \subseteq U$ of $C$.
    \item\label{item:2-smoothing-thy}  Given two smooth structures $\sigma$ and $\pi$ on $X$ extending the given smooth structure on $U \supseteq C$, there is an obstruction $\ks(\sigma,\pi) \in H^3(X, C;\Z/2)$ that  vanishes if and only if there is a neighbourhood $V \subseteq U$ of $C$ such that $\sigma$ and $\pi$ are isotopic rel.\ $V$, i.e.\ if there is a homeomorphism $f \colon X \to X$ with $f|_{V} = \Id$, such that $f^*(\pi) = \sigma$ and such that $f$ is topologically isotopic rel.~$V$ to $\Id_X$.
    \item\label{item:3-smoothing-thy} The Kirby-Siebenmann obstructions $\ks(X,C)$ and $\ks(\sigma,\pi)$ from \eqref{item:1-smoothing-thy} and \eqref{item:2-smoothing-thy} are natural for restriction to open submanifolds of $X$. More precisely, let $W$ be an open submanifold of $X$ and let $i\colon W\hookrightarrow X$ be the inclusion map. Then $i^*\colon H^4(X,C;\Z/2)\to H^4(W,W\cap C;\Z/2)$ sends $\ks(X,C)$ to $\ks(W,W\cap C)$ and $i^*\colon H^3(X,C;\Z/2)\to H^3(W,W\cap C;\Z/2)$ sends $\ks(\sigma,\pi)$ to $\ks(\sigma|_W,\pi|_W)$.
     \item\label{item:4-smoothing-thy}
         Given a smooth structure on some neighbourhood $V$ of $C$ in $X$, the Kirby-Siebenmann obstruction $\ks(X,C)$ from \eqref{item:1-smoothing-thy} is natural with respect to restriction to a neighbourhood $V'\subseteq V$ of a closed subset $C' \subseteq C$. That is, the inclusion map $H^4(X,C;\Z/2) \to H^4(X,C';\Z/2)$ sends~$\ks(X,C)$ to $\ks(X,C')$.
  \end{enumerate}
\end{theorem}

\begin{proof}
The first three items of the theorem for $\PL$ structures instead of smooth structures follows from \cite{Kirby-Siebenmann:1977-1}*{Theorem~IV.10.1} and the fact that $\TOP/\PL \simeq K(\Z/2,3)$~\cite{Kirby-Siebenmann:1977-1}*{Section~IV.10.12}. However $\PL$ 5-manifolds with smooth boundary admit a unique smooth structure up to isotopy, by smoothing theory and since $\PL/\OO$ is 6-connected~\cites{Hirsch-Mazur, Munkres-2-sphere, Cerf-V, Kervaire-Milnor:1963-1}. Hence it is legitimate to replace $\PL$ structures by smooth structures, as we have done.

The final item can be seen from the following diagram.
\[\begin{tikzcd}[row sep = small]
 V' \ar[r] \ar[dr] & V \ar[r,"T_{V}"] \ar[d] & \BO \ar[d] & \\
 & X \ar[r,"\tau_X"'] & \BTOP  \ar[r] & \BB(\TOP/\OO).
\end{tikzcd}\]
The obstructions $\ks(X,C)$ and the obstructions $\ks(X,C')$ are both  represented by the map $X \xrightarrow{\tau_X} \BTOP \to \BB(\TOP/\OO)$, and the inclusion induced map sends the former to the latter.
\end{proof}

Next we apply \cref{thm:smoothing-thy-main-thm} to deduce a naturality result for the Kirby-Siebenmann obstructions, that will be useful for submanifolds with corners.

Let $K$ be a smooth 5-manifold with corners. We refer the reader to \cite{Wall-diff-topology}*{Section~1.5} for a discussion of smooth manifolds with corners.  Suppose the corner set of $K$, denoted by $\angle K$, separates $\partial K$ into $\partial_1K$ and $\partial_2 K$. Note that $\partial_1K$ and $\partial_2K$ are smooth manifolds with boundary. Fix a smooth structure\footnote{Note that this is a \emph{smooth structure with corners}, which means that it is a maximal atlas in which two charts with corners $(U,\phi)$ and $(V,\theta)$ are smoothly compatible if $\phi\circ\theta^{-1}\colon\theta(U\cap V)\to\phi(U\cap V)$ admits a smooth extension to an open neighbourhood of each point. See~\cite{Lee00}.} $\sigma$ on a neighbourhood $U$ of $\partial K$. By \cite{Wall-diff-topology}*{Proposition~1.5.6}, $U$ contains a smooth embedding of $\partial_1K\times[0,1)$ in $K$. Let $K'$ denote the result of attaching $\partial_1 K\times[0,1)$ to $K$ by the map $h\colon\partial_1K\times\{0\}\to\partial_1K$ given by $h(x,0)=x$, and extend $\sigma|_{\partial_1 K}$ to a product structure along $[0,1)$ in $\partial_1K \times [0,1)$. Denote the resulting smooth structure on $U':=U\cup\partial_1K\times[0,1)$ by $\sigma'$. Note that $U'$ is a smooth manifold with boundary, i.e.\ no corners. Apply the same procedure to $K'$ along $\partial K'$ to obtain an unbounded manifold $K''$ and a smooth structure $\sigma''$ on $U'':=U'\cup\partial K'\times[0,1)$. See \cref{figure:K}.  Let $j\colon K\hookrightarrow K''$ be the inclusion map.

\begin{figure}[!ht]
\begin{center}
\begin{tikzpicture}
\draw (0,0) -- (3,0);
\draw (0,2) -- (3,2);
\draw (0,0) -- (0,2);
\draw (3,0) -- (3,2);
\draw (5,0) -- (8,0);
\draw (5,2) -- (8,2);
\draw (5,0) -- (5,2);
\draw (8,0) -- (8,2);
\draw[red] (5,-0.5) -- (5,0);
\draw[red] (8,-0.5) -- (8,0);
\draw[red] (5,2.5) -- (5,2);
\draw[red] (8,2.5) -- (8,2);
\draw[red, dashed] (5,2.5) -- (8,2.5);
\draw[red, dashed] (5,-0.5) -- (8,-0.5);
\draw (10.5,0) -- (13.5,0);
\draw (10.5,2) -- (13.5,2);
\draw (10.5,0) -- (10.5,2);
\draw (13.5,0) -- (13.5,2);
\draw[red] (10.5,-0.5) -- (10.5,0);
\draw[red] (13.5,-0.5) -- (13.5,0);
\draw[red] (10.5,2.5) -- (10.5,2);
\draw[red] (13.5,2.5) -- (13.5,2);
\draw[red, dashed] (10.5,2.5) -- (13.5,2.5);
\draw[red, dashed] (10.5,-0.5) -- (13.5,-0.5);
\draw[blue, dashed] (10,-0.5) -- (10.5,-0.5);
\draw[blue, dashed] (10,2.5) -- (10.5,2.5);
\draw[blue, dashed] (13.5,2.5) -- (14,2.5);
\draw[blue, dashed] (13.5,-0.5) -- (14,-0.5);
\draw[blue, dashed] (14,2.5) -- (14,-0.5);
\draw[blue, dashed] (10,2.5) -- (10,-0.5);
\draw(1.5,0.25) node{$\partial_1 K$};
\draw(1.5,1.75) node{$\partial_1 K$};
\draw(6.5,0.25) node{$\partial_1 K$};
\draw(6.5,1.75) node{$\partial_1 K$};
\draw(12,0.25) node{$\partial_1 K$};
\draw(12,1.75) node{$\partial_1 K$};
\draw(0.4,1) node{$\partial_2 K$};
\draw(2.6,1) node{$\partial_2 K$};
\draw(5.4,1) node{$\partial_2 K$};
\draw(7.6,1) node{$\partial_2 K$};
\draw(10.9,1) node{$\partial_2 K$};
\draw(13.1,1) node{$\partial_2 K$};
\end{tikzpicture}
\end{center}
\caption{The picture on the left is a schematic of $K$. The middle picture illustrates adding the collar $\partial_1 K\times[0,1)$ to obtain $K'$. The picture on the right illustrates adding $\partial K'\times[0,1)$ to obtain $K''$.}
\label{figure:K}
\end{figure}

\begin{definition}\label{defn:ks-corners}
Let $K$, $K''$, and $j$ be as above. Let $\sigma$ be a smooth structure on a neighbourhood $U$ of $\partial K$. Define the Kirby-Siebenmann obstruction $\ks(K,\partial K)$ as $j^*\ks(K'',\partial K)$ where $\ks(K'',\partial K) = \ks(K'',\partial K,U'',\sigma'')$ is the obstruction to extending the smooth structure $\sigma''$ on $U''$ to $K''$.
\end{definition}

Let $X$ be a smooth 5-manifold with boundary and let $K$ be a smooth 5-manifold with corners that is a submanifold of $X$ such that the corner set of $K$ separates $\partial K$ into $\partial_1K:=K\cap\partial X$ and~$\partial_2 K$ with $\Int \partial_2 K \subseteq \Int X$. (By definition of a submanifold, $\partial_2 K$ intersects $\partial X$ transversely.)
 Consider a smooth structure $\sigma$ on a neighbourhood $U$ of $\partial X\cup\partial K$ such that $\partial_2 K\hookrightarrow U$ is smooth; this condition guarantees the existence of a smooth bicollar neighbourhood of $\partial_2 K$ in $U$, which will be implicitly used in the proof of the next proposition.
\begin{proposition}\label{prop:naturality-corner-5dim}
Let $i\colon(K,\partial K)\hookrightarrow(X,\partial X\cup\partial K)$ be the inclusion.
The induced map \[i^*\colon H^4(X,\partial X\cup\partial K;\Z/2)\to H^4(K,\partial K;\Z/2)\] sends $\ks(X,\partial X\cup\partial K)$ to $\ks(K,\partial K)$.
\end{proposition}

\begin{proof}
Let $X'$ be the open topological manifold obtained from $X$ by attaching an exterior collar $\partial X \times [0,1)$, where $x$ in $\partial X$ is identified with $(x,0)$ in $\partial X \times [0,1)$. Extend $\sigma$, by taking a product structure along $\partial X \times [0,1)$, to $U' := U \cup \partial X \times [0,1)$, which is a neighbourhood of $\partial X\cup\partial K$ in~$X'$. Let $\sigma'$ be the result smooth structure on $U'$.
Then, by `absorbing the boundary'~\cite{Kirby-Siebenmann:1977-1}*{Proposition~IV.2.1}, this construction determines a natural bijection $\theta\colon \mathcal{S}_{\Diff}(X, \partial X\cup\partial K,\sigma)\to \mathcal{S}_{\Diff}(X', \partial X\cup\partial K,\sigma')$ and it follows from \cite{Kirby-Siebenmann:1977-1}*{Theorem~IV.10.1} and \cite{Kirby-Siebenmann:1977-1}*{Remark~IV.10.2} that $\ks(X,\partial X\cup\partial K) \mapsto \ks(X',\partial X\cup\partial K)$ under the isomorphism on $H^4(-,\partial X\cup\partial K;\Z/2)$ induced by the obvious homotopy equivalence $X' \simeq X$.
Consider the following diagram:
\[\begin{tikzcd}
H^4(X,\partial X\cup\partial K;\Z/2) \ar[rr,"\cong"] \ar[d,"i^*"] && H^4(X',\partial X\cup\partial K;\Z/2) \ar[d,"i_1^*"] \\
H^4(K,\partial K;\Z/2) & H^4(K'',\partial K;\Z/2) \ar[l,"j^*"']  & H^4(X',\partial K;\Z/2). \ar[l,"g^*"']
\end{tikzcd}\]
The map $i^*_1$ is induced by the inclusion map $i_1 \colon (X',\partial K) \to (X',\partial X \cup \partial K)$, $j^*$ is induced by the inclusion $j \colon (K,\partial K) \to (K'',\partial K)$ from \cref{defn:ks-corners},  and $g^*$ is induced by the  inclusion $(K'',\partial K)\hookrightarrow(X',\partial K)$. As per the above discussion, $\ks(X,\partial X\cup\partial K)$ is sent to $\ks(X',\partial X\cup\partial K)$ by the top horizontal map. By \cref{thm:smoothing-thy-main-thm}~\eqref{item:4-smoothing-thy}, $i^*_1\ks(X',\partial X\cup\partial K) = \ks(X',\partial K)$. It follows from \cref{thm:smoothing-thy-main-thm}~\eqref{item:3-smoothing-thy} that $g^*\ks(X',\partial K)=\ks(K'',\partial K)$. Finally by \cref{defn:ks-corners} we have that $j^*\ks(K'',\partial K)=\ks(K,\partial K)$. This concludes the proof that $i^*\ks(X,\partial X\cup\partial K) = \ks(K,\partial K)$.
\end{proof}

In practice, the manifold with corners $K$ will be either a closed tubular neighbourhood $\ol{\nu} f(Y)$ of a locally flat proper embedding $f\colon Y\to N$, or the complement $N \sm \nu f(Y)$. In fact, let $p \colon\ol{\nu} f(Y)\to f(Y)$ be a disc bundle and denote  its boundary sphere bundle by $\Sigma$.
Then $\ol{\nu} f(Y)$ is a smooth manifold with corners (note that it is not necessarily smooth in $N$) and $\angle\ol{\nu} f(Y)=p^{-1}\partial f(Y)\cap\Sigma$ separates $\partial\ol{\nu} f(Y)$ in two parts with closures $p^{-1}f(\partial Y)$ and $\Sigma$.

We also need the following more detailed characterisation of the obstruction $\ks(\sigma,\pi)$.   We write~$X_\sigma$ to denote a topological 5-manifold $X$ equipped with a smooth structure $\sigma$, and let $\pi$ be another smooth structure on $X$ that agrees with $\sigma$ near $\partial X$.

\begin{proposition}\label{prop:sigma-pi-agree-away-from-S}
Suppose that $S \subseteq X$ is a closed surface smoothly embedded in $\Int X_{\sigma}$
whose $\Z/2$-fundamental class is  Poincar\'{e} dual to $\ks(\sigma,\pi) \in H^3(X,\partial X;\Z/2)$. Then there is an arbitrarily small isotopy of~$\sigma$, supported away from $S$ and $\partial X$, to a smooth structure that agrees with $\pi$ on~$X \sm S$.
\end{proposition}

\begin{proof}
Using the inclusion $X \sm S \to X$ we have a map $H^3(X,\partial X;\Z/2) \to H^3(X \sm S,\partial X;\Z/2)$. By naturality of Kirby-Siebenmann invariants (\cref{thm:smoothing-thy-main-thm}~\eqref{item:3-smoothing-thy}), this sends the Kirby-Siebenmann invariant $\ks(\sigma,\pi)$ to the invariant of the restricted structures $\ks(\sigma|_{X \sm S},\pi|_{X \sm S})$.  We will denote restricted structures by $\sigma| := \sigma|_{X \sm S}$, and similarly for $\pi$, from now on.
 The long exact sequence of the triple $\partial X \subseteq X \sm S \subseteq X$ gives the top row of the following diagram.
\[
\begin{tikzcd}
  H^3(X,X\sm S;\Z/2) \ar[r] \ar[d,"\cong"] & H^3(X,\partial X;\Z/2) \ar[r] \ar[dd,"\cong"] & H^3(X \sm S, \partial X;\Z/2) \\
H^3(\ol{\nu}S,\partial \ol{\nu}S;\Z/2) \ar[d,"\cong"] & & \\
H_2(S;\Z/2) \ar[r] & H_2(X;\Z/2) &
\end{tikzcd}
\]
The vertical isomorphisms are given by combining homotopy invariance of homology, excision, and Poincar\'{e}-Lefschetz duality.
It follows from the diagram that the Poincar\'{e} dual to $[S] \in H_2(X;\Z)$, which by hypothesis equals  $\ks(\sigma,\pi)$, lies in the kernel of the map $H^3(X,\partial X;\Z/2) \to H^3(X \sm S,\partial X;\Z/2)$.  Thus
 $\ks(\sigma|,\pi|) =0 \in H^3(X \sm S,\partial X;\Z/2)$.

 By smoothing theory (Theorem~\ref{thm:smoothing-thy-main-thm}) there is an isotopy of $\sigma|$ to $\pi|$ on $X \sm S$ rel.\ $\partial X$.  That is, we have an isotopy of homeomorphisms $f_t \colon X\sm S \to X \sm S$, where $f_0 = \Id$, $f_t|_{\partial X}=\Id_{\partial X}$, and $f_1^*(\pi|) = \sigma|$.
 To prove the desired result we have to delve into the proof of Theorem~\ref{thm:smoothing-thy-main-thm} a little.
Such an isotopy is constructed chart by chart, and within each chart via a decomposition into handles. Then the handles are smoothed iteratively using \cite{Kirby-Siebenmann:1977-1}*{Theorem~I.3.1}.  Let $d$ be a metric on $X$. We can and shall choose charts $\{U_i\}$ covering $X \sm S$ to be such that if there exists $x \in U_i$ with $d(x,S) < \varepsilon$, then $\operatorname{diam}(U_i) < \varepsilon/10$. We can also make all charts have diameter smaller than an arbitrarily chosen global positive constant.
The construction of $f_t$ guarantees that for all~$i$, if  $x \in U_i$, then $f_t(x) \in U_i$ for all $t \in [0,1]$. It follows from this and the fact that we controlled the size of the charts as they approach $S$ that $f_t$ extends continuously to an isotopy $F_t \colon X \to X$ that fixes $S$ pointwise for all $t \in [0,1]$.   This gives the desired  isotopy of $\sigma$ to a smooth structure $\sigma'$ on $X$ such that $\sigma'|_{X \sm S} = \pi|_{X \sm S}$, i.e.\ they agree away from $S$.  Since we controlled the global size of all charts, we can also arrange for the isotopy to be arbitrarily small.
\end{proof}

\subsection{Lashof's nonsmoothable 3-knot}\label{section:lashofs-knot}

Lashof~\cite{Lashof} constructed a locally flat 3-knot $L \cong S^3 \subseteq S^5$ that is not isotopic to any smooth knot.  As observed by Kwasik and Vogel~\cites{Kwasik-Vogel,Kwasik-nonsmoothable},  Lashof's knot bounds a Seifert 4-manifold $V$ in $S^5$ with $\sign(V)/8 \equiv 1 \mod{2}$. We can use this to explain why $L$ is not smoothable. The proof is as follows.  If $L$ were smoothable, it would bound a smooth Seifert 4-manifold $V'$, which would be spin by naturality of $w_2$, and therefore would satisfy $\sign(V')/8 \equiv 0 \mod{2}$ by Rochlin's theorem. Since the signature of a Seifert 4-manifold is a knot invariant \cite{Levine:1969-1}, we arrive at a contradiction and it follows that $L$ cannot be smoothed. Since the signature is a concordance invariant, it follows also that $L$ is not concordant to any smooth $3$-knot.  Let $E_L := S^5 \sm \nu L$ be the exterior of $L$, and equip $\partial E_L \cong S^3 \times S^1$ with a standard smooth structure.

\begin{lemma}\label{lemma-ks-of-L-nonzero}
  The Kirby-Siebenmann invariant of $E_L$ satisfies \[\ks(E_L,\partial E_L) =1 \in H^4(E_L,\partial E_L;\Z/2) \cong H_1(E_L;\Z/2) \cong \Z/2.\]
\end{lemma}

\begin{proof}
If $\ks(E_L,\partial E_L)$ were trivial, then by smoothing theory there would be a smooth structure~$\tau$ on $S^5$ extending the standard smooth structure on $\ol{\nu}L$. Thus $L$ would be smooth in $\tau$.  But in fact there is a unique smooth structure on $S^5$ up to isotopy~\cite{Kervaire-Milnor:1963-1}, and hence $L$ would be isotopic to a smooth knot in the standard smooth structure on $S^5$.  Since Lashof proved this is not the case, we deduce that $\ks(E_L,\partial E_L)$ is indeed nontrivial.
\end{proof}

\section{Smoothing the complement of an embedding}\label{section:smoothing-complement}

In this section we prove the following result, which proves Step 1 from \cref{subsec:outline-of-proof}.

\begin{proposition}\label{prop:exotic-embedding}
 Let $N$ be a compact, connected, smooth 5-dimensional manifold with $($possibly empty$)$ boundary, let $Y$ be a compact 3-dimensional manifold with $($possibly empty$)$ boundary, and let $f\colon Y\to N$ be a locally flat proper topological embedding such that~$f$ is smooth near $\partial Y$. Then~$f$ is homotopic rel. boundary, via an arbitrarily small homotopy, to a smooth embedding in some smooth structure $\sigma$ on $N$ that agrees with the given smooth structure on $N$ near $\partial N$.
\end{proposition}

Let $Y=Y_1\sqcup \cdots \sqcup Y_m$, where each $Y_i$ is connected.
We write $M:= f(Y)$ and $M_i:= f(Y_i)$.
Use that $f$ is smooth near $\partial Y$ to choose a smooth normal bundle for $M$ near $\partial M$, extending a smooth normal bundle for $\partial M \subseteq \partial N$.
By Kirby-Siebenmann~\cite{KS-normal-bundles-codim-2} and Kister's theorem~\cite{Kister}, this extends to a normal $\R^2$-bundle of $M$ with structure group $\TOP(2)$.
Since $\TOP(2) \simeq \OO(2)$, we may in fact assume that this is a rank two normal vector bundle with structure group $\OO(2)$, and hence has corresponding open and closed unit-disc bundles.
Let $\nu M \subseteq N$ denote the image of the embedding of the normal bundle restricted to the open disc bundle, and let $\ol{\nu}M$ denote the image restricted to the closed disc bundle.
Let $W_f := N \sm \nu M$, $E_i:=\partial\ol{\nu} M_i \cap\partial W_f$, and define $E:= \bigcup_{i=1}^m E_i$; see \cref{figure:N}.

\begin{figure}[!ht]
\begin{center}
\begin{tikzpicture}
\draw (0,2) ellipse (0.75cm and 2cm);
\draw (4,2) ellipse (0.75cm and 2cm);
\draw (0,0) -- (4,0);
\draw (0,4) -- (4,4);
\draw[red] (0,3) ellipse (0.3cm and 0.5cm);
\draw[red] (0,1) ellipse (0.3cm and 0.5cm);
\draw[red] (4,3) ellipse (0.3cm and 0.5cm);
\draw[red] (4,1) ellipse (0.3cm and 0.5cm);
\draw[red, dashed] (0,0.5) -- (4,0.5);
\draw[red, dashed] (0,1.5) -- (4,1.5);
\draw[red, dashed] (0,2.5) -- (4,2.5);
\draw[red, dashed] (0,3.5) -- (4,3.5);
\draw[blue] (0,3) circle (0.02cm);
\draw[blue] (0,1) circle (0.02cm);
\draw[blue] (4,3) circle (0.02cm);
\draw[blue] (4,1) circle (0.02cm);
\draw[blue, dashed] (0,3) -- (4,3);
\draw[blue, dashed] (0,1) -- (4,1);
\put(50,55){\small $W_f$}
\put(70,90){\small $\ol{\nu}M_1$}
\put(70,33){\small $\ol{\nu}M_2$}
\put(25,75){\small\color{blue} $M_1$}
\put(25,17){\small\color{blue} $M_2$}
\put(25,45){\small\color{red}$E$}
\put(-10,55){\small $\partial N$}
\put(110,55){\small $\partial N$}
\end{tikzpicture}
\end{center}
\caption{A schematic diagram of $N$ decomposed as $N= W_f \cup_E \ol{\nu} M$, where $M = f(Y)$, showing the case that $Y=Y_1\sqcup Y_2$ has two connected components with nonempty boundary.}
\label{figure:N}
\end{figure}

We fix a smooth structure on a neighbourhood of $\partial N\cup E=\partial\ol{\nu} M\cup\partial W_f$. To do this, we use that $Y$, as a 3-manifold, admits an essentially unique smooth structure. Since the normal bundle of $M$ has $\OO(2)$ structure group, the total space of the normal bundle inherits a  smooth structure from that of $Y$.  The closed tubular neighbourhood $\ol{\nu} M$, which has the structure of a $D^2$-bundle $\pi \colon \ol{\nu}M \xrightarrow{} M$, therefore has the structure of a  smooth manifold with corners, with the property that $M \hookrightarrow \ol{\nu} M$ is a smooth map. The corner set gives rise to a decomposition \[\partial \ol{\nu} M = E \cup \pi^{-1}(\partial M),\] such that $E$ and $\pi^{-1}(\partial M)$ are smooth 4-manifolds with boundary.
We have a smooth structure on a collar neighbourhood of $\partial N$ in $N$. This smooth structure is compatible with the smooth structure on $\ol{\nu} M$ where they overlap, because we chose the normal bundle of $M$ using that fact that $f$ is smooth near $\partial Y$ with respect to the smooth structure on the collar neighbourhood of~$\partial N$.

Next, we will explain how to choose a topological bicollar neighbourhood of~$E$ in~$N$ together with a smooth structure that is compatible with both the smooth structure on $\ol{\nu}M$ and the given smooth structure on a neighbourhood of $\partial N$. Choose a collar neighbourhood $E \times [0,1)$ of~$E$ in~$\ol{\nu} M$ that is smooth with respect to the smooth structure on~$\ol{\nu}M$.
Then we replace $E$ with the result of pushing it halfway along the collar. Formally, we replace $E$ with $E \times \{1/2\}$, $\ol{\nu}M$ with $\ol{\nu}M \sm (E \times [0,1/2))$, and  $W_f$ with $W_f \cup (E \times [0,1/2])$, without changing notation. Now $E \times (0,1)$ gives a bicollar for $E = E \times \{1/2\}$ with a smooth structure that is compatible with the smooth structures on $\ol{\nu}M$ and near $\partial N$.

Thus we obtain a smooth structure on a neighbourhood
of $\partial N\cup E$. We obtain as well a smooth structure on a neighbourhood of $\partial W_f$ in $W_f$
such that the neighbourhood of $\partial W_f$ is a smooth manifold with corners, with corner set~$\partial\pi^{-1}(\partial M)$.

By \cref{thm:smoothing-thy-main-thm} we therefore have an obstruction \[\ks(N, \partial\ol{\nu}M\cup\partial W_f)\in H^4(N,\partial\ol{\nu}M\cup\partial W_f;\Z/2)\] to extending this smooth structure to all of $N$. It will be shown in \cref{prop:ks-lies-in-A} below that $\ks(N, \partial\ol{\nu}M\cup\partial W_f)$ is determined by \[\ks(\ol{\nu}M, \partial\ol{\nu}M)\in H^4(\ol{\nu}M, \partial\ol{\nu}M;\Z/2) \text{ and } \ks(W_f,\partial W_f)\in H^4(W_f,\partial W_f;\Z/2).\] Since the smooth structure on $E$ was obtained by restricting a structure on~$\ol{\nu} M$, it follows that $\ks(\ol{\nu} M, \partial \ol{\nu} M)=0$. Thus,~$\ks(N, \partial\ol{\nu}M\cup\partial W_f)$ is determined by $\ks(W_f,\partial W_f)$.  If the latter vanishes, then so does $\ks(N, \partial\ol{\nu}M\cup\partial W_f)$.  We also note that $\ks(N,\partial N)=0$, because~$N$ is a smooth manifold.

Our goal will therefore be to modify $f$ by a small homotopy to arrange for $\ks(W_f,\partial W_f)=0$. To begin, we record a homology computation in the next lemma.

\begin{lemma}\label{lemma:homology-computation-boundary}
The homology of $E$ satisfies $H_1(E;\Z/2) \cong \oplus^m_{i=1} (H_1(M_i;\Z/2)\oplus B_i)$, where $B_i$ is a quotient of $\Z/2$, and may depend on~$i$. If $B_i$ is nontrivial then it is generated by a meridian of~$M_i$.
\end{lemma}

\begin{proof}
 Since $S^1\hookrightarrow E_i\rightarrow M_i$ is a fibration, $M_i$ is path connected, and $\pi_1 (M_i)$ acts trivially on $H_*(S^1;\Z/2)$, we will use the Leray-Serre spectral sequence to compute $H_1(E_i;\Z/2)$. We have
\[E^2_{p,q}\cong H_p(M_i;H_q(S^1;\Z/2)) \cong \begin{cases}
H_p(M_i;\Z/2) & q=0,1 \\
0  & \mathrm{otherwise}.
\end{cases} \]
and $E^3_{p,q}=E^\infty_{p,q}$. Since the coefficient group is a field, the extension problem is trivial and
\begin{align*}
    H_1(E_i;\Z/2)&\cong E^\infty_{1,0}\oplus E^\infty_{0,1}\cong H_1(M_i;\Z/2)\oplus \Z/2/\im(d^2 \colon E^2_{2,0}\rightarrow E^2_{0,1}) \\
    &\cong H_1(M_i;\Z/2)\oplus B_i.
\end{align*}
It follows that $H_1(E;\Z/2) \cong \bigoplus^m_{i=1} H_1(M_i;\Z/2)\oplus B_i$. The $B_i$ are quotients of the terms on the $E^2$-page $H_0(M_i;H_1(S^1;\Z/2)) \cong H_1(S^1;\Z/2)$, and so if $B_i$ is nontrivial it is generated by a meridian to $M_i$, as asserted.
\end{proof}

Whether or not $B_i$ is trivial depends on the differential $d^2$.
It will not be important for our later proofs whether $B_i$ is nontrivial, and so we do not include an investigation of this.

 Let $A \subseteq H^4(W_f,\partial W_f;\Z/2)$ be the subgroup generated by $\{\PD^{-1}[\mu_i]\}_{i=1}^m$, where $[\mu_i] \in H_1(W_f;\Z/2)$ is the class represented by a meridian to $M_i$, and $\PD$ denotes the Poincar\'{e}-Lefschetz duality isomorphism. That is, writing $\iota \colon E \to W_f$ for the inclusion map, $A$ is by definition the subgroup of $H^4(W_f,\partial W_f;\Z/2)$ Poincar\'{e} dual to $\oplus_{i=1}^m \iota(B_i) \subseteq H_1(W_f;\Z/2)$.

\begin{proposition}\label{prop:ks-lies-in-A}
The Kirby-Siebenmann obstruction
$\ks(W_f,\partial W_f) \in H^4(W_f,\partial W_f;\Z/2)$ determines
$\ks(N, \partial\ol{\nu}M\cup\partial W_f)\in H^4(N,\partial\ol{\nu}M\cup\partial W_f;\Z/2)$.
Moreover, $\ks(W_f,\partial W_f)$ lies in the subgroup~$A$.
\end{proposition}

\begin{proof}
All homology and cohomology in this proof will be with $\Z/2$ coefficients, and so to save space we omit them from the notation.
Decompose the pair $(N, \partial\ol{\nu}M\cup\partial W_f)$
as \[( \ol{\nu}M, \partial\ol{\nu}M) \cup (W_f,\partial W_f).\]
The intersections are $\ol{\nu}M \cap W_f = E =  \partial\ol{\nu}M \cap \partial W_f$.  Consider the relative cohomology Mayer-Vietoris sequence~\cite{Hatcher}*{p.~204}:
\[\cdots \to H^{n-1}(E,E) \to H^n(N, \partial\ol{\nu}M\cup\partial W_f)\to H^n( \ol{\nu}M, \partial\ol{\nu}M) \oplus H^n(W_f,\partial W_f)\to H^n(E,E)\to\cdots\]
Taking $n=4$ and observing that $H^i(E,E) =0$ for all $i$, we deduce that \[H^4(N, \partial\ol{\nu}M\cup\partial W_f)\cong H^4( \ol{\nu}M, \partial\ol{\nu}M)\oplus H^4(W_f,\partial W_f)\]
where this isomorphism has coordinates the two restrictions to $(\ol{\nu}M, \partial\ol{\nu}M)$ and $(W_f,\partial W_f)$. Therefore, by \cref{prop:naturality-corner-5dim} applied twice, once to the inclusion~$\ol{\nu}M \hookrightarrow N$ and once to the inclusion~$W_f \hookrightarrow N$, this isomorphism sends $\ks(N, \partial\ol{\nu}M\cup\partial W_f) \in H^4(N, \partial\ol{\nu}M\cup\partial W_f)$ to
\begin{align*}
   (\ks(\ol{\nu}M, \partial\ol{\nu}M), \ks(W_f,\partial W_f)) = (0, \ks(W_f,\partial W_f)) \in H^4( \ol{\nu}M, \partial\ol{\nu}M)\oplus H^4(W_f,\partial W_f).
\end{align*}
Here we use that $\ks(\ol{\nu}M, \partial\ol{\nu}M)=0$, which as mentioned above holds because our chosen smooth structure on $\partial\ol{\nu}M$ was obtained by restricting a structure on $\ol{\nu} M$. This proves the first statement of the proposition.

To prove the second sentence, consider the following diagram:
\[\begin{tikzcd}[column sep = tiny]
 H^3(\partial\ol{\nu}M\cup\partial W_f,\partial N) \ar[r] \ar[d,"\cong"] & H^4(N,\partial\ol{\nu}M\cup\partial W_f) \ar[r,"j^*"] \ar[d,"\cong"] & H^4(N,\partial N) \\
H^3(E,\partial E)\ar[r] \ar[d,"\cong"] & H^4( \ol{\nu}M, \partial\ol{\nu}M)\oplus H^4(W_f,\partial W_f) \ar[d,"\cong"] \\
H_1(E) \ar[r,"k"] & H_1(\ol{\nu}M)\oplus H_1(W_f)  &
\end{tikzcd}\]
where the upper row is an excerpt from the cohomology long exact sequence of the triple $\partial N\subseteq \partial\ol{\nu}M\cup\partial W_f\subseteq N$, the top left vertical isomorphism is by excision and the bottom vertical isomorphisms use Poincar\'{e}--Lefschetz duality.
Let $(0, \gamma)\in H_1(\ol{\nu}M)\oplus H_1(W_f)$ be the Poincar\'{e}-Lefschetz dual of
\[(\ks(\ol{\nu}M, \partial\ol{\nu}M), \ks(W_f,\partial W_f))= (0,\ks(W_f,\partial W_f)).\]
Since \[j^*(\ks (N, \partial\ol{\nu}M\cup\partial W_f))=\ks(N, \partial N)=0,\]
by \cref{thm:smoothing-thy-main-thm}~\eqref{item:4-smoothing-thy} and the fact that $N$ is smooth, it follows from exactness of the top row and commutativity of the diagram that $(0,\gamma) \in \im k$.
The map $k \colon H_1(E) \to H_1(\ol{\nu}M) \oplus H_1(W_f)$ is induced by the inclusions $\kappa_1 \colon E \hookrightarrow \ol{\nu}M$ and $\kappa_2 \colon E \hookrightarrow W_f$.
By \cref{lemma:homology-computation-boundary}, \[H_1(E) \cong \oplus^m_{i=1} (H_1(M_i)\oplus B_i) \cong H_1(M) \oplus^m_{i=1} B_i.\]
Let \[(\alpha,\beta_1,\dots,\beta_m) \in H_1(M) \oplus^m_{i=1} B_i\] be such that $k(\alpha,\beta_1,\dots,\beta_m) = (0,\gamma) \in H_1(\ol{\nu}M)\oplus H_1(W_f)$.
Note that $\kappa_1|_{B_i} =0$ and $\kappa_1|_{H_1(M)}$ is an isomorphism. The first of these observations implies that $\kappa_1(\alpha,\beta_1,\dots,\beta_m) = \kappa_1|_{H_1(M)}(\alpha)$.  Since $\kappa_1(\alpha,\beta_1,\dots,\beta_m)=0$, we have that $\kappa_1|_{H_1(M)}(\alpha)=0$. Using that $\kappa_1|_{H_1(M)}$ is an isomorphism, we deduce that $\alpha=0$. Thus $\PD(\ks(W_f,\partial W_f)) =  (0,\beta_1,\dots,\beta_m) \in \im (\oplus_{i=1}^m B_i)$ is a sum of meridians of the connected components $M_i$ of~$M$. It follows that $\ks(W_f,\partial W_f)$ lies in $A$, as desired.
\end{proof}

Let $L \cong S^3 \subseteq S^5$ denote Lashof's non-smoothable 3-knot (see \cref{section:lashofs-knot}).
Write \[\ks(W_f,\partial W_f) = \sum_{i=1}^m a_i (\PD^{-1}[\mu_i]),\] for $a_i \in \Z/2$ defined by this equality and the stipulation that we take $a_i=0$ if $\PD^{-1}[\mu_i]=0$ in $H^4(W_f,\partial W_f;\Z/2)$.
 If $a_i =1$ then we form a connected sum $M_i \# L$ in an arbitrarily small $5$-ball, while if $a_i =0$ we leave $M_i$ alone. Let $g\colon Y\hookrightarrow N$ denote the resulting embedding.
Define \[\A := \{i \mid a_i =1\} \subseteq \{1,\dots,m\}.\]
Let $W_g := N \sm \nu g(Y)$. Note that \[W_g\cong W_f\cup_{\sqcup_{\A} (S^1\times D^3)_i} \bigsqcup_{\A} E_{L_i}.\]
That is, $W_f$ and $\sqcup_i E_{L_i}$ attached along $\sqcup_{\A} (S^1\times D^3)_i$, where $(S^1\times {0})_i$ is identified with a meridian to $M_i$ and a meridian to $L_i$, for each $i \in \A$, and we extend to a tubular neighbourhood $(S^1 \times D^3)_i$ in $\partial W_g$ and $\partial E_{L_i}$ respectively.
Also, note that
\[\partial W_g= \Big(\partial W_f \sm \sqcup_{\A} (S^1 \times \mathring{D}^3)_i\Big) \cup_{\sqcup_{\A} (S^1 \times S^2)_i} \Big( \bigsqcup_{\A} \partial E_{L_i} \sm \sqcup_{\A} (S^1 \times \mathring{D}^3)_i\Big)\]
 Hence,  $\partial W_g\cup(\sqcup_{\A} (S^1\times D^3)_i)$ decomposes as $\partial W_f\cup \bigcup_{\A} \partial E_{L_i}$.
\cref{figure:W-g} shows an illustration of~$W_g$ when one Lashof knot is attached.

\begin{figure}[h!]
\begin{center}
\begin{tikzpicture}
\draw (0,1) -- (0,0) -- (2,0) -- (2,2) -- (0,2) -- (0,1) -- (2,1);
\put(-20,30){\small\color{blue}$W_g$}
\put(25,10){\small\color{red}$W_f$}
\put(25,45){\small\color{red}$E_{L}$}
\put(15,30){\small\color{red}$S^1\times D^3$}
\put(60,10){\small\color{red}$\partial W_f$}
\put(60,40){\small\color{red}$\partial E_L$}
\end{tikzpicture}
\end{center}
\caption{A schematic diagram of $W_g$ when one Lashof knot is attached.}
\label{figure:W-g}
\end{figure}

\begin{proposition}\label{prop:-vanishing-ks-for-C-plus-L}
 We have that $\ks(W_g,\partial W_g) =0 \in H^4(W_g,\partial W_g;\Z/2)$.
\end{proposition}

\begin{proof}
 All homology and cohomology in this proof will be with $\Z/2$ coefficients, and so to save space we omit them from the notation.
 Recall that $\ks(W_f,\partial W_f) = \sum_{\A} \PD^{-1}[\mu_i] \in H^4(W_f,\partial W_f)$.
From the relative cohomology Mayer-Vietoris sequence of the pair \[(W_g,\partial W_g\cup(\sqcup_{\A} (S^1\times D^3)_i))=(W_f, \partial W_f)\cup(\sqcup_{\A} E_{L_i},\sqcup_{\A}\partial E_{L_i}),\] using that $W_f\cap(\sqcup_{\A} E_{L_i})=\partial W_f\cap(\sqcup_{\A}\partial E_{L_i})=\sqcup_{\A} (S^1\times D^3)_i$, we get via an argument similar to that in the proof of \cref{prop:ks-lies-in-A},  that \[H^4 (W_g,\partial W_g\cup(\sqcup_{\A} (S^1\times D^3)_i))\cong H^4(W_f, \partial W_f)\oplus_{\A} H^4 (E_{L_i},\partial E_{L_i}).\]
Hence the image of $\ks(W_g,\partial W_g\cup(\sqcup_{\A} (S^1\times D^3)_i))$ under this isomorphism is \[(\ks(W_f, \partial W_f),\ks (E_{L_1},\partial E_{L_1}),\dots,\ks (E_{L_k},\partial E_{L_k})) \in H^4(W_f, \partial W_f)\oplus_{\A} H^4 (E_{L_i},\partial E_{L_i}).\]
Recall that by \cref{lemma-ks-of-L-nonzero} we have that $\ks (E_{L_i},\partial E_{L_i}) =1$ for each $i \in \A$, represented by~$\PD^{-1}[\mu_{L_i}]$, the Poincar\'{e} dual to a meridian of $L_i$.

Consider the following diagram:
\[\begin{tikzcd}[column sep = tiny]
  H^3(\partial W_g\cup(\sqcup_{\A} (S^1\times D^3))_i,\partial W_g) \ar[r] \ar[d,"\cong"] & H^4(W_g, \partial W_g\cup(\sqcup_{\A} (S^1\times D^3)_i)) \ar[r,"j^*"] \ar[d,"\cong"] & H^4(W_g,\partial W_g)\ar[d,"="] \\
H^3(\sqcup_{\A} (S^1\times D^3)_i,\sqcup_{\A} (S^1\times S^2)_i)\ar[r] \ar[d,"\cong"] & H^4(W_f, \partial W_f)\oplus_{\A} H^4(E_{L_i},\partial E_{L_i})\ar[r]\ar[d,"\cong"] & H^4(W_g,\partial W_g)\ar[d,"\cong"]\\
H_1(\sqcup_{\A} (S^1\times D^3)_i)\cong\oplus_{\A} H_1((S^1\times {0})_i) \ar[r,"{\Phi}"] & H_1(W_f)\oplus_{\A} H_1(E_{L_i})\ar[r] & H_1(W_g) &
\end{tikzcd}\]
The upper row is an excerpt from the cohomology long exact sequence of the triple \[\partial W_g\subseteq \partial W_g\cup(\sqcup_{\A} (S^1\times D^3))_i\subseteq W_g,\] the top left vertical isomorphism is by excision and the bottom vertical isomorphisms use Poincar\'{e}--Lefschetz duality.
By naturality of the Kirby-Siebenmann obstruction (\cref{prop:naturality-corner-5dim} applied twice), the upper middle vertical isomorphism sends \[\kappa := \ks(W_g, \partial W_g\cup(\sqcup_{\A} (S^1\times D^3)_i))\] to  \[\big(\ks(W_f, \partial W_f),\ks (E_{L_1},\partial E_{L_1}),\dots,\ks (E_{L_k},\partial E_{L_k})\big) \in H^4(W_f, \partial W_f)\oplus_{\A} H^4(E_{L_i},\partial E_{L_i}).\]
On the other hand by \cref{thm:smoothing-thy-main-thm}~\eqref{item:4-smoothing-thy},
$j^*(\kappa) = \ks(W_g,\partial W_g) \in H^4(W_g,\partial W_g)$.
By commutativity of the top right square it follows that  \[\big(\ks(W_f, \partial W_f),\ks (E_{L_1},\partial E_{L_1}),\dots,\ks (E_{L_k},\partial E_{L_k})\big)\] maps under the right hand map of the middle row to \[\ks(W_g,\partial W_g) \in H^4(W_g,\partial W_g).\]
By commutativity of the bottom right square of the diagram, the Poincar\'{e}-Lefschetz dual of the former, $(\sum_{\A}[\mu_i],\sum_{\A}[\mu_{L_i}]) \in  H_1(W_f)\oplus_{\A} H_1(E_{L_i})$, is sent to
\[\gamma := \PD\big(\ks (W_g,\partial W_g)\big) \in H_1(W_g),\]
the Poincar\'{e}-Lefschetz dual of $\ks (W_g,\partial W_g)\in H^4(W_g,\partial W_g)$.
Note that the bottom row is the Mayer-Vietoris homology sequence of the decomposition $W_f\cup_{\A} E_{L_i}=W_g$, and for each $i \in \A$ we have $\Phi([S^1\times {0}]_i)=([\mu_i], [\mu_{L_i}])$, where $[S^1\times {0}]_i$ is the generator of $H_1((S^1\times {0})_i)$.
Hence by linearity, \[\Phi\big(\sum_{\A} [S^1\times {0}]_i\big)=\big(\sum_{\A} [\mu_i],\sum_{\A}[\mu_{L_i}]\big).\]
Thus $\gamma=0$ by exactness, and since $\PD$ is an isomorphism it follows that $\ks (W_g,\partial W_g)=0$.
\end{proof}

The main result of this section follows.

\begin{proof}[Proof of Proposition~{\ref {prop:exotic-embedding}}]
Since  $\ks(W_g,\partial W_g) =0$, we can extend the standard smooth structure on $\partial N \cup \ol{\nu} g(Y)$ to all of $N$. Call the resulting smooth structure $\sigma$. By construction, $g(Y)$ is smooth in~$\sigma$, and~$\sigma$ agrees with the given smooth structure of $N$ near $\partial N$.
Each connected sum of~$M_i$ with~$L_i$ can be done arbitrarily close to $M_i=f(Y_i)$, so we can assume that we altered $f$ by an arbitrarily small homotopy.
\end{proof}

\section{Comparing with the standard smooth structure on \texorpdfstring {$N$}{N}}\label{section:comparing-with-std}

Next, we need to compare the smooth structure $\sigma$ we have just constructed with the given smooth structure $\std$ on $N$. The submanifold $g(Y)$ is smooth in $\sigma$, but is a priori not smooth in~$\std$.
 We aim to reduce to a finite collection of local problems, namely neighbourhoods $V_i \subseteq \Int N$ where~$g(Y)$ need not be smooth in~$\std$.  Then we will apply the argument that all 2-knots are smoothly slice \cites{Kervaire-slice-knots,Sunuk} to further modify $g(Y)$ in each of these neighbourhoods $V_i$, replacing~$g(Y) \cap V_i$ with a slice disc for $g(Y) \cap \partial V_i \cong S^2$ that is smooth in the structure $\std$.   Our aim is the following proposition, which proves Step 2 from the introduction.
The combination of Proposition~\ref{prop:exotic-embedding} and Proposition~\ref{prop:smooth-embedding-body} proves Theorem~\ref{thm:main-intro}.

\begin{proposition}\label{prop:smooth-embedding-body}
 Let $N$ be a compact, connected, smooth 5-dimensional manifold with $($possibly empty$)$ boundary, let $Y$ be a compact 3-dimensional manifold with $($possibly empty$)$ boundary, and let $g\colon Y\to N_\sigma$ be a smooth embedding for some $\sigma$ such that $\sigma$ and $\std$ agree near $\partial N$.  Then $g$ is homotopic rel. boundary, via an arbitrarily small homotopy, to a smooth embedding in $N_{\std}$.
\end{proposition}

To begin, recall that the structures $\sigma$ and $\std$ correspond via smoothing theory (Theorem~\ref{thm:smoothing-thy-main-thm}) to two lifts $\sigma, \std\colon N \to \BO$ of $\tau_{N} \colon N \to \BTOP$. The difference between these lifts gives rise to a map $N \to \TOP/\OO$, and whence to an element  $\ks(\sigma,\std) \in H^3(N,\partial N ; \Z/2) \cong H_2(N;\Z/2)$. In this section we redefine $M:= g(Y)$

\begin{lemma}\label{lemma:represent-ks-by-S}
  The class $\PD(\ks(\sigma,\std)) \in H_2(N;\Z/2)$ can be represented by a closed surface $S \subseteq \Int N$, which is smoothly embedded in $\sigma$ and is transverse to $M:= g(Y)$.
\end{lemma}

\begin{proof}
We consider the group $\mathcal{N}_2(N)$ of unoriented surfaces mapping to $N$, up to bordism.  The Atiyah-Hirzebruch spectral sequence for this has $E^2$-page
\[E^2_{p,q} \cong H_p(N;\mathcal{N}_q).\]
The unoriented bordism groups are given~\cite{Thom-cobordism} in the range $q \in \{0,1,2\}$ by $\mathcal{N}_0 \cong \Z/2 \cong \mathcal{N}_2$, and $\mathcal{N}_1=0$.  Using that the $q=1$ row on the $E^2$-page consists entirely of zeros, we have an exact sequence
\[H_3(N;\Z/2) \xrightarrow{d^3_{3,0}} H_0(N;\Z/2) \to\mathcal{N}_2(N) \to H_2(N;\Z/2) \to 0.\]
In particular every element of $H_2(N;\Z/2)$ lifts to $\mathcal{N}_2(N)$, and so can be represented by a map~$h \colon \Sigma \to N$ from some closed surface~$\Sigma$ into~$N$.

By \cite{Hirsch-diff-top}*{Theorem~2.2.6} we can approximate $h$ by a smooth map in $[N]_{\sigma}$, and by \cite{Hirsch-diff-top}*{Theorems~2.2.12~and~2.2.14} we can approximate the result by an embedding, $h' \colon \Sigma \to N$. We write~$S:= h'(\Sigma)$. Since both $S$ and $M$ are smooth in $\sigma$, we  apply transversality to complete the proof.
\end{proof}

By Proposition~\ref{prop:sigma-pi-agree-away-from-S}, by an arbitrarily small isotopy of $\sigma$ away from $S$ and $\partial N$,
and hence of~$M \cap (N \sm S)$, we can assume that the smooth structures $\sigma$ and $\std$ agree in the complement of  the surface $S$.  Replace $M$ and $\sigma$ by the outcomes of this isotopy.

Let $\nu S$ denote a smooth open tubular neighbourhood of $S$ in the smooth structure $\sigma$.
We have that $M\sm \nu S$ is smooth in $[N \sm \nu S]_{\std}$.
By compactness and transversality, $S \pitchfork M$ consists of finitely many points, $p_1,\dots,p_n$ say. Moreover, the intersection $M \cap \partial \ol{\nu} S$ consists of a copy of~$S^2$ for each point $p_i \in S \pitchfork M$, which bounds a 3-ball $D^3_i \subseteq M \cap \ol{\nu} S$ with the centre of~$D^3_i$ equal to~$p_i$.
In fact the intersection $M \cap \ol{\nu} S$ comprises exactly $\bigcup_{i=1}^n D^3_i$;  the~$D^3_i$ are pairwise disjoint.

 Since $D^3_i$ is locally flat and codimension 2, it has a normal bundle~\cite{KS-normal-bundles-codim-2}.
We take a normal bundle of each $D^3_i$ in $\ol{\nu} S$.
We obtain an inclusion of pairs
\[(D^3_i \times \R^2,S^2 \times \R^2) \subseteq (\ol{\nu} S, \partial \ol{\nu} S).\]
Pull back the smooth structure $\std$ to this to obtain
\[V_i := [D^3_i \times \R^2]_{\std}. \]
This $V_i$ is a smooth manifold that is homeomorphic to $D^3 \times \R^2$, with boundary $\partial V_i$ identified with $S^2 \times \R^2$ with its usual product smooth structure.
In the boundary, $M \cap \partial V_i$ is a smooth 2-sphere $T_i$ that is identified with $S^2 \times \{0\} \subseteq S^2 \times \R^2$.

We remark that $(V_i,\partial V_i)$ may not be diffeomorphic rel.\ boundary to $(D^3 \times \R^2,S^2 \times \R^2)$. In addition while $M \cap V_i$ is smooth in $\sigma$, this need not be the case in $\std$.

\begin{lemma}\label{lemma:finding:Y-s}
  The 2-sphere $T_i \subseteq \partial V_i$ bounds a compact, orientable 3-manifold $Z_i$ smoothly embedded in $V_i = [D^3_i \times \R^2]_{\std}$.
\end{lemma}

\begin{proof}
Let $g \colon V_i \xrightarrow{\cong} D^3 \times \R^2$ be the homeomorphism from the definition of $V_i$, which is a diffeomorphism near $\partial V_i$, and let
$V_i' := g^{-1}(D^3 \times D^2)$ be the corresponding closed disc bundle.
 Consider the sequence of maps
  \[f \colon  V_i \xrightarrow{g} D^3 \times \R^2 \xrightarrow{\pr_2} \R^2.\]
Then $f$ is smooth near $\partial V_i$, and $f|_{\partial V_i}^{-1}(\{0\}) = T_i$. Also note that $f^{-1}(0) = g^{-1}(D^3 \times \{0\}) \subseteq V_i'$.

Perturb $f$ rel.\  a neighbourhood of $\partial V_i$ so it becomes smooth and transverse to $0 \in \R^2$.
By making the perturbation sufficiently small, we can and will assume that $Z_i := f^{-1}(0)$ still lies in~$V_i' \subseteq V_i$.

We have a smooth 3-manifold $Z_i \subseteq V_i' \subseteq V_i$ with boundary $T_i \subseteq \partial V_i$.   As the inverse image of a closed set, $Z_i$ is closed. Since  $Z_i \subseteq V_i'$ and~$V_i'$ is compact, $Z_i$ is compact.  Since $V_i$ is orientable, $0 = w_1(V_i) = w_1(Z_i) + w_1(\nu Z_i)$, and hence $w_1(Z_i) = w_1(\nu Z_i)$.  However~$w_1(\nu Z_i)=0$ because~$\nu Z_i$ can be obtained as the pull back of the (trivial) normal bundle of $\{0\} \subseteq \R^2$. It follows that~$Z_i$ is orientable.  We have therefore  constructed the  compact, orientable, smoothly embedded 3-manifold $Z_i \subseteq V_i$ we desire.
\end{proof}

Now we can prove the Proposition~\ref{prop:smooth-embedding-body}, which is the goal of this section.

\begin{proof}[Proof of Proposition~\ref{prop:smooth-embedding-body}]
  To prove the proposition, it remains to find, for each $i=1,\dots,n$, a smooth slice disc $D^3 \subseteq V_i$ with $\partial D^3 = T_i = S^2 \times \{0\} \subseteq S^2 \times \R^2 = \partial V_i$.  By Lemma~\ref{lemma:finding:Y-s}, for each $i$ we have a smooth, compact, orientable 3-manifold $Z_i$ with $\partial Z_i = T_i$. Since $Z_i$ is orientable and 3-dimensional, it is parallelisable, and thus is in particular spin. We now apply the argument of Sunukjian~\cite{Sunuk} from his Section 5 and the proof of his Theorem 6.1. As mentioned in the introduction, this is similar to and was inspired by Kervaire's theorem~\cite{Kervaire-slice-knots} that every 2-knot is slice.  For the convenience of the reader we give an outline here.

  First perform ambient 1-surgeries on $Z_i$ to arrange that $\pi_1(V_i \sm Z_i)$ is cyclic.   By \cite{Sunuk}*{Proposition~5.1~and~Lemma~5.2}, there is a spin structure on $Z_i$ such that every spin structure preserving surgery on $Z_i$ can be performed ambiently.  Here we use that $\pi_1(V_i \sm Z_i)$ is cyclic, so that every circle in $Z_i$ bounds an embedded 2-disc whose interior lies in $V_i \sm Z_i$.
Using this spin structure,  the union $Z_i \cup D^3$ is a closed, smooth, spin 3-manifold.  The group $\Omega_3^{\operatorname{Spin}} =0$, so $Z_i \cup D^3$ is spin null-bordant.  By \cite{Sunuk}*{Lemma~5.4}, there is a sequence of spin structure compatible surgeries on circles in  $Z_i$ that convert it to $D^3$.  Perform these surgeries ambiently, and obtain a smoothly embedded $D^3 \subseteq [V_i]_{\std}$, as desired, in the restriction to $V_i$ of the smooth structure~$\std$.

Replacing $M \cap V_i$ with this 3-ball, for each $i$, yields a smooth embedding $g'\colon Y\hookrightarrow N$ in the smooth structure $\std$.
Since $V_i$ is contractible, $\pi_3(V_i,T_i) \cong \pi_2(T_i) \cong \Z$, and so there is a unique homotopy class of a 3-ball in $V_i$ with boundary mapping to $T_i$ via a homeomorphism.  Thus we changed $g$ by a homotopy.
By making the $V_i$ as small as we please,  we can further arrange that we changed $g$ by an arbitrarily small homotopy.
\end{proof}

As mentioned above, Propositions~\ref{prop:exotic-embedding} and \ref{prop:smooth-embedding-body} combine to complete the proof of Theorem~\ref{thm:main-intro}, noting that in both cases all the modifications we made to the embedding, from $f$ to $g$ to $g'$, consisted of local homotopies or isotopies, in all cases supported outside a neighbourhood of $\partial N$.

\section{Conditions for smoothing up to isotopy}\label{section:Jae-choon-suggestions}

As shown by Lashof's 3-knot~\cite{Lashof} (Section~\ref{section:lashofs-knot}), it is not in general possible to isotope a locally flat embedding of a 3-manifold to a smooth embedding. Our main result shows this is possible with an arbitrarily small homotopy.
Here we discuss the extent to which smoothing up to isotopy is possible.

As above let  $Y=Y_1\sqcup \cdots \sqcup Y_m$ be a compact 3-manifold with connected components $Y_i$, and let~$N$ be a compact, connected, smooth 5-manifold.
We will use the Kirby-Siebenmann invariant $\ks(W_f,\partial W_f) \in H^4(W_f,\partial W_f;\Z/2)$ of the exterior $W_f := N \sm \nu f(Y)$, and we will use the relative Kirby-Siebenmann invariant $\ks(\sigma,\std) \in H^3(N,\partial N;\Z/2)$ comparing the smooth structure $\sigma$ on $N$ arising from Step 1 (Proposition~\ref{prop:exotic-embedding}) with the given smooth structure $\std$ on $N$.  These invariants were recalled in detail in Section~\ref{section:smoothing-theory}.
In practice, these invariants are not always easy to evaluate. One way to do this for \eqref{it:i-smoothing-up-to-isotopy}  could be to use the ideas of Kwasik and Vogel~\cites{Kwasik-Vogel,Kwasik-nonsmoothable} discussed in Section~\ref{section:lashofs-knot} to relate $\ks(W_f,\partial W_f)$ to the signature of an appropriate 4-manifold.

\begin{scholium}\label{schol:isotopy-criteria}
Let $f\colon Y\to N$ be a locally flat proper topological embedding that is smooth near~$\partial Y$.
\begin{enumerate}[(i)]
  \item\label{it:i-smoothing-up-to-isotopy} If $\ks(W_f,\partial W_f)=0$ then there exists a smooth structure $\sigma$ on $N$ with respect to which $f$ is smooth.
  \item If in addition $\langle \ks(\sigma,\std),[f(Y_i)]\rangle = 0 \in \Z/2$ for each connected component $Y_i$ of $Y$, then $f$ is topologically isotopic rel.\ boundary, via an arbitrarily small isotopy, to a smooth embedding.
\end{enumerate}
\end{scholium}

\begin{proof}
If $\ks(W_f,\partial W_f)=0$, then Step 1 (Proposition~\ref{prop:exotic-embedding}) can be completed without connect summing with any Lashof knots. We obtain a smooth structure $\sigma$ on $N$ in which $f$ is smooth, that agrees with the standard smooth structure on $N$ near $\partial N$. Now suppose $\langle \ks(\sigma,\std),[f(Y_i)]\rangle = 0$ for each $i=1,\dots,m$. Let~$S$ be an embedded surface Poincar\'{e} dual to $\ks(\sigma,\std)$ that intersects~$f(Y)$ transversely (such an $S$ was produced in Lemma~\ref{lemma:represent-ks-by-S}). The condition implies, by intersection theory, that for each~$i$ the count of transverse intersection points between $S$ and~$f(Y_i)$ is even. For every~$i$, tube~$S$ to itself, along  $f(Y_i)$, to obtain a new surface $S'$, in the same $\Z/2$-homology class, $[S] = [S'] \in H_2(N;\Z/2)$, and such that $S' \cap f(Y) = \emptyset$. It then follows from Proposition~\ref{prop:sigma-pi-agree-away-from-S} that $f$ is isotopic to a smooth embedding in $\std$.
\end{proof}

\def\MR#1{}
\bibliography{bib}

@article {Kister,
    AUTHOR = {Kister, James M.},
     TITLE = {Microbundles are fibre bundles},
   JOURNAL = {Ann. of Math. (2)},
  FJOURNAL = {Annals of Mathematics. Second Series},
    VOLUME = {80},
      YEAR = {1964},
     PAGES = {190--199}
}

@article {Kirby-smoothing-lf-embeddings,
    AUTHOR = {Kirby, Robion C.},
     TITLE = {Smoothing locally flat imbeddings of differentiable manifolds},
   JOURNAL = {Topology},
  FJOURNAL = {Topology. An International Journal of Mathematics},
    VOLUME = {6},
      YEAR = {1967},
     PAGES = {207--220}
}

@article {Wall-locally-flat-PL-smoothing-codim-2,
    AUTHOR = {Wall, C. Terence C.},
     TITLE = {Locally flat {${\rm PL}$} submanifolds with codimension two},
   JOURNAL = {Proc. Cambridge Philos. Soc.},
  FJOURNAL = {Proceedings of the Cambridge Philosophical Society},
    VOLUME = {63},
      YEAR = {1967},
     PAGES = {5--8}}

@Article{Venema-1987,
 Author = {Venema, Gerard A.},
 Title = {Approximating codimension two embeddings of cells},
 FJournal = {Pacific Journal of Mathematics},
 Journal = {Pac. J. Math.},
 Volume = {126},
 Number = {1},
 Pages = {165--195},
 Year = {1987}}

@incollection {Kervaire-higher-dim-knots,
    AUTHOR = {Kervaire, Michel A.},
     TITLE = {On higher dimensional knots},
 BOOKTITLE = {Differential and {C}ombinatorial {T}opology ({A} {S}ymposium
              in {H}onor of {M}arston {M}orse)},
     PAGES = {105--119},
 PUBLISHER = {Princeton Univ. Press, Princeton, NJ},
      YEAR = {1965}}

@book{Lee00,
    AUTHOR = {Lee, John M.},
     TITLE = {Introduction to smooth manifolds},
    SERIES = {Graduate Texts in Mathematics},
    VOLUME = {218},
   EDITION = {Second},
 PUBLISHER = {Springer, New York},
      YEAR = {2013},
     PAGES = {xvi+708},
      ISBN = {978-1-4419-9981-8}
}

@Article{Hsiang-Levine-Szczarba,
 Author = {Hsiang, Wu Chung and Levine, Jerome P. and Szczarba, Robert H.},
 Title = {On the normal bundle of a homotopy sphere embedded in {Euclidean} space},
 FJournal = {Topology},
 Journal = {Topology},
 ISSN = {0040-9383},
 Volume = {3},
 Pages = {173--181},
 Year = {1965}}

@inproceedings{SchultzSmoothableSO,
  title={Smoothable submanifolds of a smooth manifold},
  author={Reinhard Schultz},
  note={Available at \url{https://math.ucr.edu/~res/miscpapers/smoothablesubmflds.pdf}}
}

@article {Lashof-Rothenberg,
    AUTHOR = {Lashof, Richard K. and Rothenberg, Melvin},
     TITLE = {Microbundles and smoothing},
   JOURNAL = {Topology},
  FJOURNAL = {Topology. An International Journal of Mathematics},
    VOLUME = {3},
      YEAR = {1965},
     PAGES = {357--388}}

@article {Kuga,
    AUTHOR = {Kuga, Ken'ichi},
     TITLE = {Representing homology classes of {$S\sp{2}\times S\sp{2}$}},
   JOURNAL = {Topology},
  FJOURNAL = {Topology. An International Journal of Mathematics},
    VOLUME = {23},
      YEAR = {1984},
    NUMBER = {2},
     PAGES = {133--137}}

@article {Rudolph,
    AUTHOR = {Rudolph, Lee},
     TITLE = {Some topologically locally-flat surfaces in the complex
              projective plane},
   JOURNAL = {Comment. Math. Helv.},
  FJOURNAL = {Commentarii Mathematici Helvetici},
    VOLUME = {59},
      YEAR = {1984},
    NUMBER = {4},
     PAGES = {592--599}}

@article {LW-90,
    AUTHOR = {Lee, Ronnie and Wilczy\'{n}ski, Dariusz M.},
     TITLE = {Locally flat {$2$}-spheres in simply connected
              {$4$}-manifolds},
   JOURNAL = {Comment. Math. Helv.},
  FJOURNAL = {Commentarii Mathematici Helvetici},
    VOLUME = {65},
      YEAR = {1990},
    NUMBER = {3},
     PAGES = {388--412}
}

@article {Luo,
    AUTHOR = {Luo, Feng},
     TITLE = {Representing homology classes of {${\bf C}{\rm
              P}^2\#\;\overline{{\bf C}{\rm P}}{}^2$}},
   JOURNAL = {Pacific J. Math.},
  FJOURNAL = {Pacific Journal of Mathematics},
    VOLUME = {133},
      YEAR = {1988},
    NUMBER = {1},
     PAGES = {137--140}}

@inproceedings {KS-normal-bundles-codim-2,
    AUTHOR = {Kirby, Robion C. and Siebenmann, Laurent C.},
     TITLE = {Normal bundles for codimension {$2$} locally flat imbeddings},
 BOOKTITLE = {Geometric topology ({P}roc. {C}onf., {P}ark {C}ity, {U}tah,
              1974)},
    SERIES = {Lecture Notes in Math., Vol. 438},
     PAGES = {310--324},
 PUBLISHER = {Springer, Berlin},
      YEAR = {1975}}

@article {Thom-cobordism,
    AUTHOR = {Thom, Ren\'{e}},
     TITLE = {Quelques propri\'{e}t\'{e}s globales des vari\'{e}t\'{e}s diff\'{e}rentiables},
   JOURNAL = {Comment. Math. Helv.},
  FJOURNAL = {Commentarii Mathematici Helvetici},
    VOLUME = {28},
      YEAR = {1954},
     PAGES = {17--86}}

@Article{Sunuk,
 Author = {Sunukjian, Nathan S.},
 Title = {Surfaces in 4-manifolds: concordance, isotopy, and surgery},
 FJournal = {IMRN. International Mathematics Research Notices},
 Journal = {Int. Math. Res. Not.},
 ISSN = {1073-7928},
 Volume = {2015},
 Number = {17},
 Pages = {7950--7978},
 Year = {2015}}

@book {Hirsch-diff-top,
    AUTHOR = {Hirsch, Morris W.},
     TITLE = {Differential topology},
    SERIES = {Graduate Texts in Mathematics},
    VOLUME = {33},
      NOTE = {Corrected reprint of the 1976 original},
 PUBLISHER = {Springer-Verlag, New York},
      YEAR = {1994},
     PAGES = {x+222},
 }

@unpublished{Cha-Kim,
  note = {Preprint, available at ar{x}iv:2303.12857},
  author = {Cha, Jae Choon and Kim, Byeorhi},
  title = {Light bulb smoothing for topological surfaces in 4-manifolds},
  year = {2023},
}

@Article{Klug-Miller,
 Author = {Klug, Michael R. and Miller, Maggie},
 Title = {Concordance of surfaces in 4-manifolds and the {Freedman}-{Quinn} invariant},
 FJournal = {Journal of Topology},
 Journal = {J. Topol.},
 ISSN = {1753-8416},
 Volume = {14},
 Number = {2},
 Pages = {560--586},
 Year = {2021}}

@Article{ST-FQ,
 Author = {Schneiderman, Rob and Teichner, Peter},
 Title = {Homotopy versus isotopy: spheres with duals in 4-manifolds},
 FJournal = {Duke Mathematical Journal},
 Journal = {Duke Math. J.},
 ISSN = {0012-7094},
 Volume = {171},
 Number = {2},
 Pages = {273--325},
 Year = {2022}
}

@Article{Schwartz-I,
 Author = {Schwartz, Hannah R.},
 Title = {Equivalent non-isotopic spheres in 4-manifolds},
 FJournal = {Journal of Topology},
 Journal = {J. Topol.},
 ISSN = {1753-8416},
 Volume = {12},
 Number = {4},
 Pages = {1396--1412},
 Year = {2019}}

@unpublished{AMY,
  note = {Preprint, available at ar{x}iv:2109.14578},
  author = {Audoux, Benjamin and  Meilhan, Jean-Baptiste and  Yasuhara, Akira},
  title = {Milnor concordance invariant for knotted surfaces and beyond},
  year = {2021},
}

@Article{Stong-uniqueness,
 Author = {Stong, Richard},
 Title = {Uniqueness of {{\(\pi_ 1\)}}-negligible embeddings in 4-manifolds: {A} correction to theorem 10.5 of {Freedman} and {Quinn}},
 FJournal = {Topology},
 Journal = {Topology},
 Volume = {32},
 Number = {4},
 Pages = {677--699},
 Year = {1993}}

@article {Lashof,
    AUTHOR = {Lashof, Richard K.},
     TITLE = {A nonsmoothable knot},
   JOURNAL = {Bull. Amer. Math. Soc.},
  FJOURNAL = {Bulletin of the American Mathematical Society},
    VOLUME = {77},
      YEAR = {1971},
     PAGES = {613--614}
}

@article {Kwasik-Vogel,
    AUTHOR = {Kwasik, S{\l}awomir and Vogel, Pierre},
     TITLE = {On invariant knots},
   JOURNAL = {Math. Proc. Cambridge Philos. Soc.},
  FJOURNAL = {Mathematical Proceedings of the Cambridge Philosophical
              Society},
    VOLUME = {96},
      YEAR = {1984},
    NUMBER = {3},
     PAGES = {473--475}}

@article {Kwasik-nonsmoothable,
    AUTHOR = {Kwasik, S{\l}awomir},
     TITLE = {On invariant nonsmoothable knots},
   JOURNAL = {Amer. J. Math.},
  FJOURNAL = {American Journal of Mathematics},
    VOLUME = {109},
      YEAR = {1987},
    NUMBER = {1},
     PAGES = {157--164}}

@article {Kervaire-slice-knots,
    AUTHOR = {Kervaire, Michel A.},
     TITLE = {Les n\oe uds de dimensions sup\'{e}rieures},
   JOURNAL = {Bull. Soc. Math. France},
  FJOURNAL = {Bulletin de la Soci\'{e}t\'{e} Math\'{e}matique de France},
    VOLUME = {93},
      YEAR = {1965},
     PAGES = {225--271}}

@article {Milnor-microbundles,
    AUTHOR = {Milnor, John},
     TITLE = {Microbundles. {I}},
   JOURNAL = {Topology},
  FJOURNAL = {Topology. An International Journal of Mathematics},
    VOLUME = {3},
NOTE = {Suppl.\ 1},
      YEAR = {1964},
     PAGES = {53--80}
}

@inproceedings {Milnor-ICM,
    AUTHOR = {Milnor, John},
     TITLE = {Topological manifolds and smooth manifolds},
 BOOKTITLE = {Proc. {I}nternat. {C}ongr. {M}athematicians ({S}tockholm,
              1962)},
     PAGES = {132--138},
 PUBLISHER = {Inst. Mittag-Leffler, Djursholm},
      YEAR = {1963},
   MRCLASS = {57.30 (57.20)},
  MRNUMBER = {0161345},
}

@article {Munkres-smoothing,
    AUTHOR = {Munkres, James},
     TITLE = {Obstructions to the smoothing of piecewise-differentiable
              homeomorphisms},
   JOURNAL = {Ann. of Math. (2)},
  FJOURNAL = {Annals of Mathematics. Second Series},
    VOLUME = {72},
      YEAR = {1960},
     PAGES = {521--554},
      ISSN = {0003-486X},
   MRCLASS = {57.00},
  MRNUMBER = {121804},
MRREVIEWER = {P. Dedecker},
       DOI = {10.2307/1970228},
       URL = {https://doi.org/10.2307/1970228},
}

@article {Munkres-smoothing-II,
    AUTHOR = {Munkres, James},
     TITLE = {Higher obstructions to smoothing},
   JOURNAL = {Topology},
  FJOURNAL = {Topology. An International Journal of Mathematics},
    VOLUME = {4},
      YEAR = {1965},
     PAGES = {27--45}
}

@article {Munkres-smoothing-III,
    AUTHOR = {Munkres, James},
     TITLE = {Concordance is equivalent to smoothability},
   JOURNAL = {Topology},
  FJOURNAL = {Topology. An International Journal of Mathematics},
    VOLUME = {5},
      YEAR = {1966},
     PAGES = {371--389},
      ISSN = {0040-9383},
   MRCLASS = {57.05},
  MRNUMBER = {214081},
MRREVIEWER = {Morris W. Hirsch},
       DOI = {10.1016/0040-9383(66)90029-2},
       URL = {https://doi.org/10.1016/0040-9383(66)90029-2},
}

@article {Munkres-smoothing-IV,
    AUTHOR = {Munkres, James},
     TITLE = {Obstructions to imposing differentiable structures},
   JOURNAL = {Illinois J. Math.},
  FJOURNAL = {Illinois Journal of Mathematics},
    VOLUME = {8},
      YEAR = {1964},
     PAGES = {361--376},
      ISSN = {0019-2082},
   MRCLASS = {57.10},
  MRNUMBER = {180979},
MRREVIEWER = {M. A. Kervaire},
       URL = {http://projecteuclid.org/euclid.ijm/1256059559},
}

@article {Munkres-2-sphere,
    AUTHOR = {Munkres, James},
     TITLE = {Differentiable isotopies on the {$2$}-sphere},
   JOURNAL = {Michigan Math. J.},
  FJOURNAL = {Michigan Mathematical Journal},
    VOLUME = {7},
      YEAR = {1960},
     PAGES = {193--197},
      ISSN = {0026-2285},
   MRCLASS = {57.10 (57.20)},
  MRNUMBER = {144354},
MRREVIEWER = {P. Dedecker},
       URL = {http://projecteuclid.org/euclid.mmj/1028998426},
}

@article {Hirsch-smoothing,
    AUTHOR = {Hirsch, Morris W.},
     TITLE = {Obstruction theories for smoothing manifolds and maps},
   JOURNAL = {Bull. Amer. Math. Soc.},
  FJOURNAL = {Bulletin of the American Mathematical Society},
    VOLUME = {69},
      YEAR = {1963},
     PAGES = {352--356}
}

@article {Cairns,
    AUTHOR = {Cairns, Stewart S.},
     TITLE = {The manifold smoothing problem},
   JOURNAL = {Bull. Amer. Math. Soc.},
  FJOURNAL = {Bulletin of the American Mathematical Society},
    VOLUME = {67},
      YEAR = {1961},
     PAGES = {237--238}
}

@incollection {Cerf-I,
    AUTHOR = {Cerf, Jean},
     TITLE = {La nullit\'{e} de {$\pi _{0}({\rm Diff}S^{3})$}. 1. {P}osition
              du probl\`eme},
 BOOKTITLE = {S\'{e}minaire {H}enri {C}artan, 1962/63},
     PAGES = {Exp. 9-10, 27},
 PUBLISHER = {Secr\'{e}tariat math\'{e}matique, Paris},
      YEAR = {1964}
}

@incollection {Cerf-III,
    AUTHOR = {Cerf, Jean},
     TITLE = {La nullit\'{e} de {$\pi _{0}(\mathrm{Diff}S^{3})$}. 3.
              {C}onstruction d'une section pour le rev\^{e}tement {${\mathcal{R}}$}},
 BOOKTITLE = {S\'{e}minaire {H}enri {C}artan, 1962/63},
     PAGES = {Exp. 21, 25},
 PUBLISHER = {Secr\'{e}tariat math\'{e}matique, Paris},
      YEAR = {1964}
}

@incollection {Cerf-II,
    AUTHOR = {Cerf, Jean},
     TITLE = {La nullit\'{e} de {$\pi _{0}({\rm Diff}\,S^{3})$}. 2. {E}spaces
              fonctionnels li\'{e}s aux d\'{e}compositions d'une sph\`{e}re plong\'{e}e dans
              {R}3},
 BOOKTITLE = {S\'{e}minaire {H}enri {C}artan, 1962/63},
     PAGES = {Exp. 20, 29},
 PUBLISHER = {Secr\'{e}tariat math\'{e}matique, Paris},
      YEAR = {1964}
}

@article {Cerf-VI,
    AUTHOR = {Cerf, Jean},
     TITLE = {Groupes d'automorphismes et groupes de diff\'{e}omorphismes des
              vari\'{e}t\'{e}s compactes de dimension {$3$}},
   JOURNAL = {Bull. Soc. Math. France},
  FJOURNAL = {Bulletin de la Soci\'{e}t\'{e} Math\'{e}matique de France},
    VOLUME = {87},
      YEAR = {1959},
     PAGES = {319--329}
}

@book {Cerf-V,
    AUTHOR = {Cerf, Jean},
     TITLE = {Sur les diff\'{e}omorphismes de la sph\`ere de dimension trois
              {$(\Gamma _{4}=0)$}},
    SERIES = {Lecture Notes in Mathematics, No. 53},
 PUBLISHER = {Springer-Verlag, Berlin-New York},
      YEAR = {1968},
     PAGES = {xii+133}}

@incollection {Cerf-IV,
    AUTHOR = {Cerf, Jean},
     TITLE = {La nullit\'{e} de {$\pi _{0}({\rm Diff}S^{3})$}. {T}h\'{e}or\`emes de
              fibration des espaces de plongements. {A}pplications},
 BOOKTITLE = {S\'{e}minaire {H}enri {C}artan, 1962/63},
     PAGES = {Exp. 8, 13},
 PUBLISHER = {Secr\'{e}tariat math\'{e}matique, Paris},
      YEAR = {1964}}

@unpublished{CP,
  note = {Preprint, available at ar{X}iv:2009.13461},
  author = {Conway, Anthony and Powell, Mark},
  title = {Embedded surfaces with infinite cyclic knot group},
  year = {2020},
}

@Article{Behrens-Hill-Hopkins-Mahowald,
 Author = {Mark {Behrens} and Michael {Hill} and Michael J. {Hopkins} and Mark {Mahowald}},
 Title = {{Detecting exotic spheres in low dimensions using coker \(J\)}},
 FJournal = {{Journal of the London Mathematical Society. Second Series}},
 Journal = {{J. Lond. Math. Soc., II. Ser.}},
 ISSN = {0024-6107},
 Volume = {101},
 Number = {3},
 Pages = {1173--1218},
 Year = {2020},
 Publisher = {John Wiley \& Sons, Chichester; London Mathematical Society, London}}

@book {CLM,
    AUTHOR = {L\"uck, Wolfgang and Macko, Tibor},
     TITLE = {Surgery theory---foundations},
    SERIES = {Grundlehren der mathematischen Wissenschaften},
    VOLUME = {362},
      NOTE = {With contributions by Diarmuid Crowley},
 PUBLISHER = {Springer, Cham},
      YEAR = {2024},
     PAGES = {xv+956}}

@article{Moise52,
	author = {Moise, Edwin},
	fjournal = {Annals of Mathematics. Second Series},
	journal = {Ann. of Math. (2)},
	pages = {96--114},
	title = {Affine structures in {$3$}-manifolds. {V}. {T}he triangulation theorem and {H}auptvermutung},
	volume = {56},
	year = {1952}}

@book{Hirsch-Mazur,
	Author = {Hirsch, Morris W. and Mazur, Barry},
	Publisher = {Princeton University Press, Princeton, N. J.; University of Tokyo Press, Tokyo},
	Title = {Smoothings of piecewise linear manifolds},
	Year = {1974}}

@article{Kervaire-Milnor:1963-1,
	Author = {Kervaire, Michel~A.~ and Milnor, John},
	Journal = {Ann. of Math. (2)},
	Pages = {504--537},
	Title = {Groups of homotopy spheres. {I}},
	Volume = {77},
	Year = {1963}}

@article{Levine:1969-1,
	Author = {Levine, Jerome~P.~},
	Journal = {Comment. Math. Helv.},
	Pages = {229--244},
	Title = {Knot cobordism groups in codimension two},
	Volume = {44},
	Year = {1969}}

@book{Wall-diff-topology,
	Author = {Wall, C. Terence C.},
	Pages = {viii+346},
	Publisher = {Cambridge University Press, Cambridge},
	Series = {Cambridge Studies in Advanced Mathematics},
	Title = {Differential topology},
	Url = {https://doi.org/10.1017/CBO9781316597835},
	Volume = {156},
	Year = {2016}}

@book{Kirby-Siebenmann:1977-1,
	Address = {Princeton, N.J.},
	Author = {Kirby, Robion C. and Siebenmann, Laurence C.},
	Note = {With notes by John Milnor and Michael Atiyah, Annals of Mathematics Studies, No. 88},
	Pages = {vii+355},
	Publisher = {Princeton University Press},
	Title = {Foundational essays on topological manifolds, smoothings, and triangulations},
	Year = {1977}}

@article{Whdj1961a,
	author = {J. Henry C. Whitehead},
	journal = {{Ann. of Math.}},
	pages = {154--212},
	title = {{Manifolds with transverse fields in euclidean space}},
	volume = {{(2) 73}},
	year = {1961}}

@article{Brown62,
	Author = {Brown, Morton},
	Doi = {10.2307/1970177},
	Fjournal = {Annals of Mathematics. Second Series},
	Issn = {0003-486X},
	Journal = {Ann. of Math. (2)},
	Mrclass = {54.78},
	Mrnumber = {0133812},
	Mrreviewer = {E. Michael},
	Pages = {331--341},
	Title = {Locally flat imbeddings of topological manifolds},
	Url = {http://dx.doi.org/10.2307/1970177},
	Volume = {75},
	Year = {1962}}

@book{Hatcher,
	Author = {Hatcher, Allen},
	Pages = {xii+544},
	Publisher = {Cambridge University Press, Cambridge},
	Title = {Algebraic topology},
	Year = {2002}}

\end{document}